\documentclass[11pt]{amsart}
\usepackage{url, 
  amssymb,setspace, mathrsfs,fontenc, comment}
 \usepackage[alphabetic]{amsrefs}
\usepackage{fullpage} 
\usepackage{color}
\usepackage{tikz-cd}
\usepackage[all]{xy}
\usepackage{amsmath}

\usepackage{enumitem}
\usepackage{url, 
  amssymb,setspace, mathrsfs,fontenc}
\usepackage{amsrefs}
\usepackage{graphicx}
\usepackage{verbatim}
\usepackage{hyperref}
\hypersetup{backref=true}
\usepackage{mathtools}
\usepackage{tikz}
\usetikzlibrary{chains}

\tikzset{node distance=2em, ch/.style={circle,draw,on chain,inner sep=2pt},chj/.style={ch,join},every path/.style={shorten >=4pt,shorten <=4pt},line width=1pt,baseline=-1ex}

\newtheorem{thm}{Theorem}
\newtheorem{lem}[thm]{Lemma}
\newtheorem{prop}[thm]{Proposition}
\newtheorem{conj}[thm]{Conjecture}
\newtheorem{cor}[thm]{Corollary}
\newtheorem{defe}[thm]{Definition}

\theoremstyle{remark}
\newtheorem{rem}[thm]{Remark}

\DefineSimpleKey{bib}{myurl}
\newcommand\myurl[1]{\url{#1}}
\BibSpec{webpage}{
  +{}{\PrintAuthors} {author}
  +{,}{ \textit} {title}
  +{}{ \parenthesize} {date}
  +{,}{ \myurl} {myurl}
}
\usepackage{arydshln}

\newcommand{\nc}{\newcommand}

\nc{\ssec}{\subsection}

\nc{\on}{\operatorname}

\nc{\Ai}{\mathrm{Ai}}
\nc{\sE}{\mathscr{E}}
\nc{\sF}{\mathscr{F}}
\nc{\sL}{\mathscr{L}}
\nc{\sD}{\mathscr{D}}
\nc{\sA}{\mathscr{A}}

\nc{\cC}{\mathcal{C}}
\nc{\cG}{\mathcal{G}}
\nc{\cV}{\mathcal{V}}
\nc{\cK}{{k(\!(s)\!)}}

\nc{\cE} {\mathcal{E}}
\nc{\cI}{\mathcal{I}}
\nc{\cO}{\mathcal{O}}
\nc{\cF}{\mathcal{F}}
\nc{\cZ}{\mathcal{Z}}
\nc{\cD}{\mathcal{D}}
\nc{\cDt}{\mathcal{D}^\times}
\nc{\cH}{\mathcal{H}}
\nc{\bG}{\mathbb{G}}
\nc{\bZ}{\mathbb{Z}}
\nc{\bQ}{\mathbb{Q}}
\nc{\bR}{\mathbb{R}}
\nc{\bC}{\mathbb{C}}
\nc{\bQl}{\overline{\mathbb{Q}}_\ell}
\nc{\bQlt}{\bQl^\times} 
\nc{\FG}{\mathrm{FG}}
\nc{\dR}{\mathrm{dR}}

\nc{\uG}{\underline{G}}
\nc{\cB}{\mathcal{B}}
\nc{\cU}{\mathcal{U}}
\nc{\rat}{\mathrm{rat}}
\nc{\Hyp}{\mathscr{Hyp}}
\nc{\Lie}{\mathrm{Lie}}
      
\nc{\fF}{\mathfrak{F}}
\nc{\fB}{\mathfrak{B}}
\nc{\fZ}{\mathfrak{Z}}
\nc{\fx}{\mathfrak{x}}
\nc{\fy}{\mathfrak{y}}
\nc{\fb}{\mathfrak{b}}
\nc{\fk}{\mathfrak{k}}
\nc{\fI}{\mathfrak{i}}
\nc{\fj}{\mathfrak{j}}
\nc{\fg}{\mathfrak{g}}
\nc{\fu}{\mathfrak{u}}
\nc{\fl}{\mathfrak{l}}
\nc{\fn}{\mathfrak{n}}
\nc{\cP}{\mathcal{P}}
\nc{\ft}{\mathfrak{t}}
\nc{\fz}{\mathfrak{z}}
\nc{\fc}{\mathfrak{c}}
\nc{\fh}{\mathfrak{h}}
\nc{\fp}{\mathfrak{p}}
\nc{\bone}{\mathbf{1}}
\nc{\tg}{\mathtt{g}}
\nc{\hfg}{\widehat{\fg}}
\nc{\hG}{\widehat{G}}
\nc{\hg}{\hat{\mathfrak{g}}}
\nc{\Ug}{\hat{U}(\mathfrak{g})}

\nc{\bGm}{\mathbb{G}_m}
\nc{\bGa}{\mathbb{G}_a}
\nc{\bL}{\mathbf{L}}
\nc{\bF}{\mathbb{F}}
\nc{\bK}{\mathbf{K}}
\nc{\bJ}{\mathbf{J}}
\nc{\bI}{\mathbf{I}}
\nc{\bV}{\mathbb{V}}
\nc{\bP}{\mathbb{P}}
\nc{\bA}{\mathbb{A}}
\nc{\bN}{\mathbb{N}}

\nc {\Q}{\mathrm{Q}}
\nc{\diag}{\mathrm{diag}}
\nc{\ev}{\mathrm{ev}}
\nc{\Res}{\mathrm{Res}}
\nc{\Fl}{\mathcal{F}\ell}
\nc{\Ad}{\mathrm{Ad}}
\nc{\ad}{\mathrm{ad}}
\nc{\pr}{\mathrm{pr}}
\nc{\Sl}{\mathfrak{sl}}
\nc{\gl}{\mathfrak{gl}}
\nc{\ra}{\rightarrow}
\nc{\tra}{\twoheadrightarrow}
\nc{\hra}{\hookrightarrow}
\nc{\quo}{\mathopen{ /\!/}}
\nc{\GL}{\mathrm{GL}}
\nc{\SL}{\mathrm{SL}}
\nc{\Sp}{\mathrm{Sp}}
\nc{\SO}{\mathrm{SO}}
\nc{\PGL}{\mathrm{PGL}}
\nc{\Bun}{\mathrm{Bun}}
\nc{\supp}{\mathrm{supp}}
\nc{\bgamma}{\bar{\gamma}}
\nc{\I}{\mathrm{I}}
\nc{\II}{\mathrm{II}}
\nc{\III}{\mathrm{III}}
\nc{\ab}{\mathrm{ab}}
\nc{\td}{\mathrm{d}}
\nc{\Ht}{\mathrm{ht}}
\nc{\der}{\mathrm{der}}

\nc{\St}{\mathrm{St}}

\nc         {\rar}[1]       {\stackrel{#1}{\longrightarrow}}

\nc{\fa}{\mathfrak{a}}
\nc{\Hitch}{\mathrm{Hitch}}

\nc{\RS}{\mathrm{RS}}

\nc{\tp}{\mathfrak{p}}
\nc{\cA}{\mathcal{A}}
\nc{\cN}{\mathcal{N}}
\nc{\cW}{\mathcal{W}}

\nc{\opp}{\mathrm{opp}}
\nc{\Ind}{\mathrm{Ind}}
\nc{\sAn}{\mathrm{can}}
\nc{\Vac}{\mathrm{Vac}}
\nc{\Op}{\mathrm{Op}}
\nc{\Lg}{\check{\fg}}
\nc{\LV}{\check{V}}
\nc{\Lh}{\check{h}}
\nc{\LG}{\check{G}}
\nc{\Spec}{\mathrm{Spec}}
\nc{\End}{\mathrm{End}}

\nc{\rX}{\mathring{X}}
\nc{\ru}{\mathring{u}}

\nc{\sW}{\mathscr{W}}
\nc{\sH}{\mathscr{H}}
\nc{\sV}{\mathscr{V}}
\nc{\geom}{\mathrm{geom}}
\nc{\Irr}{\mathrm{Irr}}
\nc{\fm}{\mathfrak{m}}
\nc{\aff}{\mathrm{aff}}
\nc{\Aut}{\mathrm{Aut}}
\nc{\cJ}{\mathcal{J}}
\nc{\fs}{\mathfrak{s}}
\nc{\Stab}{\mathrm{Stab}}
\nc{\tw}{{\widetilde{w}}}
\nc{\gen}{\mathrm{gen}}
\nc{\genn}{\mathrm{genn}}
\nc{\sss}{\mathrm{ss}}
\nc{\spp}{\mathrm{sp}}
\nc{\Hom}{\mathrm{Hom}}
\nc{\bm}{\mathbf{m}}
\nc{\HG}{\mathcal{HG}}
\nc{\Gal}{\mathrm{Gal}}

\nc{\tP}{\mathtt{P}}
\nc{\tL}{\mathtt{L}}
\nc{\tU}{\mathtt{U}}

\nc{\tW}{\widetilde{W}}
\nc{\Hk}{\on{Hk}}
\nc{\cL}{\mathcal{L}}
\nc{\talpha}{\widetilde{\alpha}}
\nc{\tQ}{{\widetilde{Q}}}
\nc{\ochi}{\overline{\chi}}
\nc{\tdelta}{\widetilde{\Delta}}
\nc{\wt}{\mathrm{wt}}
\nc{\Id}{\mathrm{Id}}
\nc{\id}{\mathrm{id}}
\nc{\fQ}{\mathfrak{Q}}

\nc{\Rep}{\mathrm{Rep}}
\nc{\Hecke}{\mathrm{Hecke}}
\nc{\Gr}{\mathrm{Gr}}
\nc{\GR}{\mathrm{GR}}
\nc{\IC}{\mathrm{IC}}
\nc{\Std}{\mathrm{Std}} 
\nc{\Db}{\mathrm{D}^{\mathrm{b}}}
\nc{\tr}{\mathrm{tr}}
\nc{\Hit}{\mathrm{Hit}}
\nc{\gr}{\mathrm{gr}}
\nc{\frJ}{\mathfrak{J}}

\nc{\Perv}{\mathrm{Perv}}
\nc{\hl}{\overleftarrow{h}}
\nc{\hr}{\overrightarrow{h}}
\nc{\bfT}{\mathbf{T}}
\nc{\upH}{\mathrm{H}}
\nc{\BunJ}{\mathrm{Bun}_{J}}
\nc{\Bunpl}{\mathrm{Bun}_{J^{+}}}
\nc{\AS}{\mathrm{AS}}

\nc{\lan}{\langle}
\nc{\ran}{\rangle}

\begin{document} 
\renewcommand{\thepart}{\Roman{p|art}}

\renewcommand{\partname}{\hspace*{20mm} Part}

\address{Department of Mathematics, Massachusetts Institute of Technology} 
\email{kjakob@mit.edu} 

\address{School of Mathematics and Physics, The University of Queensland} 
\email{masoud@uq.edu.au}

\address{Department of Mathematics, University of Minnesota} 
\email{yilingfei12@gmail.com}

\title{Airy sheaves for reductive groups} 
\author{Konstantin Jakob, Masoud Kamgarpour and Lingfei Yi}
\date{\today} 
\maketitle

\begin{abstract} We construct a class of $\ell$-adic local systems on $\bA^1$ that generalizes the Airy sheaves defined  by N. Katz to reductive groups. These sheaves are finite field analogues of generalizations of the classical Airy equation $y''(z)=zy(z)$. We employ the geometric Langlands correspondence to construct the sought-after local systems as eigenvalues of certain rigid Hecke eigensheaves, following the methods developed by Heinloth, Ngô and Yun. The construction is motivated by a special case of Adler and Yu’s construction of tame supercuspidal representations. The representations that we consider can be viewed as deeper analogues of simple supercuspidals. For $\GL_n$, we compute the Frobenius trace of the local systems in question and show that they agree with Katz's Airy sheaves. We make precise conjectures about the ramification behaviour of the local systems at $\infty$. These conjectures in particular imply cohomological rigidity of Airy sheaves.
\end{abstract}

\tableofcontents

\section{Introduction}
\subsection{Background} 
The classical Airy equation  
\begin{equation}\label{eq:AiryEquation} 
y''(z)=zy(z)
\end{equation} 
is one of the simplest non-trivial complex ordinary differential equations with an irregular singular point at $z=\infty$. Its local behaviour around that point is well-studied and exhibits the Stokes phenomenon, cf. \cite{Stokes, BoalchTwisted, BoalchStokesTop}. The Airy equation defines a connection on the trivial vector bundle over $\bA^1$ of rank two and this connection is rigid
in the sense that up to gauge equivalence it is determined by the formal structure at $\infty$ \cite{KatzRigid, BlochEsnault}.  Rigidity of the Airy equation follows for example from the fact that it is the Fourier transform of a rank one differential equation of exponential type.  Such a realisation allowed Katz to generalize this type of equation and define  analogous $\ell$-adic local systems \cite{Katz87, KatzGroups}.

\subsubsection{Katz's $\ell$-adic Airy sheaves} 
Let $k=\bF_q$ be a finite field,  $\psi : k \to \bQlt$  a non-trivial character, and $\cL_{\psi}$ the associated Artin--Schreier sheaf. Let $f(z)\in k[z]$ be a polynomial of degree $n+1$ and  $\cL_{\psi}(f)$  the pull-back of $\cL_{\psi}$ along the map $\bA^{1}_{k}\rightarrow \bA^{1}_{k}$ defined by $f$. 

\begin{defe}  \label{d:AirySheaf} 
The Fourier transform $\cF_f(\cL_{\psi}(f))$ is called the \emph{rank $n$ Airy sheaf}.  Replacing $k$ with $\mathbb{C}$ and $\cL_\psi$ with the exponential sheaf, we obtain the \emph{rank $n$ Airy connection}.\footnote{When $n=2$, we recover the connection corresponding to the classical Airy equation (\ref{eq:AiryEquation}).}
\end{defe} 

The Airy local systems have been studied extensively by Katz who proved, in particular, that they are wildly ramified at $\infty$ of slope  $(n+1)/n$ and rigid. The study of the global monodromy group of these sheaves has been a subject of active research by Katz and collaborators, cf.  \cite{Katz87, KatzGroups, KatzFinite}.

\subsubsection{Airy $\hG$-connections} 
Now let $\hG$ be a simple complex algebraic group. The following connections were introduced in \cite{KS21}: 
\begin{defe}  The \emph{Airy $\hG$-connection} is the connection on the trivial bundle on $\mathbb{A}^1$ defined by 
\begin{equation}\label{eq:AiryConnection}  
\nabla := d-(N^{-}+zE_{\theta} )dz.
\end{equation} 
\end{defe} 
Here, $N^{-}$ is a principal nilpotent element in $\widehat{\fg} = \Lie(\hG)$ and $E_{\theta}$ is a generator of the highest root space such that $N^{-}+E_{\theta}$ is semisimple. We can think of the above connection as a ``deeper" version of the Frenkel--Gross connection \cite{FrenkelGross}; indeed, on $\bP^1\setminus\{0,\infty\}$ we may change the coordinate to $t=z^{-1}$. Then the Airy connection takes the form 
\begin{equation}
\nabla = d+(N^{-}+t^{-1}E_{\theta} )\frac{dt}{t^2}.
\end{equation}
Changing $t^2$ in the denominator to $t$, we obtain the Frenkel--Gross connection. It was proved in \cite{KS21} that Airy $G$-connections are rigid and have an irregular singularity at $\infty$ with slope $\frac{h+1}{h}$, where $h$ is the Coxeter number of $\hG$. When $\hG=\SL_n$, we recover Katz's rank $n$ Airy connection.

\subsection{Overview of our main result} The goal of this article is to use the geometric Langlands correspondence to find the $\ell$-adic counterpart of Airy $\hG$-connections. Equivalently, our goal is to construct reductive analogues of Katz's $\ell$-adic Airy sheaf. To realise our goal, we use the remarkable ideas of Heinloth, Ngô, and Yun \cite{HNY13}, later axiomatized by Z. Yun in the framework of ``rigid automorphic data" \cite{Yun14}.

Let $G$ be a reductive group over $k$ that is geometrically almost simple with Langlands dual $\hG$, $K$ the function field of $\bP^{1}$, $F$ the completion of the local ring of $\bP^1$ at $\infty$ and $\bA_{K}$ the ring of adèles. Roughly speaking in our situation a rigid automorphic datum is a pair $(J,\chi)$ where $J\subset G(F)$ is a subgroup of finite codimension in a parahoric subgroup and $\chi$ is a character on $K$ factoring through a finite-dimensional (over $k$) quotient $J/J^{+}$ such that there is a small but non-zero number of irreducible cuspidal automorphic representations $\pi$ of $G(\bA_{K})$ such that $\pi_{\infty}$ has an eigenvector under $J$ on which $J$ acts through the given character $\chi$.  

The choice of $J$ is informed by the explicit local Langlands correspondence in the tame situation. In particular the construction of certain supercuspidal representations by Adler \cite{Ad98} and Yu \cite{Yu01}  informed our choice of level group. In more details,  they construct supercuspidal representations in an explicit way as compact inductions from certain compact open subgroups of $G(F)$. These compact open subgroups are essentially the level group we are looking for. In \S \ref{sss:localAiry} we give a more detailed explanation of our motivation for the choice of the level group.

Let $\Bun_{J^+}$ be the moduli stack of $G$-bundles with $J^+$-level structure at $\infty$. As discussed  in \S \ref{ss:components}, we have a canonical bijection between $\pi_0(\Bun_{J^+})$ and the set of minuscule coweights $M\subset X_*(T)$.  Our main result can be formulated as follows: 

\begin{thm} 
\begin{enumerate} 
\item[(i)] For each $\mu\in M$, the connected component $\Bun_{J^+}^\mu$ supports a unique non-zero $(J,\chi)$-equivariant perverse sheaf $\cA_\mu$. 
\item[(ii)] The perverse sheaf $\displaystyle \cA= \bigsqcup_{\mu \in M} \cA_\mu$ can be endowed with a structure of a Hecke eigensheaf.
\item[(iii)] The corresponding Hecke eigenvalue $\mathrm{Ai}_{\hG}$ satisfies the property that $\mathrm{Ai}_{\hG,V}$ is a semisimple local system for each $V$ in $\Rep(\hG)$.
\end{enumerate} 
\end{thm}

\begin{defe} The Hecke eigenvalue $\mathrm{Ai}_{\hG}$ is called the \emph{$\hG$-Airy local system}. 
\end{defe} 

By comparing Frobenius traces, we prove that when $G=\GL_{n}$ and $V$ is the standard representation, $\mathrm{Ai}_{\hG,V}$ is isomorphic to a rank $n$ $\ell$-adic Airy local system. This justifies considering $\mathrm{Ai}_{\hG}$ as the reductive analogues of Katz's Airy local systems 

\begin{rem}\label{rem:characteristic} We use the assumption $p > h$ for the following reasons: 
\begin{enumerate} \item It guarantees the existence of an $\Ad$-invariant non-degenerate symmetric bilinear form $\kappa$ on $\fg$, cf. \cite[\S B.6]{JantzenLie}.
\item It allows us to work with the cyclic grading of order $h$ on $\fg$ defined by the point $\check{\rho}/h$ and use the  Vinberg--Levy theory, cf. \cite{Levy}.
\item It guarantees the existence of a Springer isomorphism between the nilpotent variety of $\fg$ and the unipotent variety of $G$ that serves as an analogue of the exponential map, cf. \cite{Sobaje}. Under our assumption this Springer isomorphism restricts to the reduction mod $p$ of the exponential map on the unipotent radical of any parabolic subgroup $P$. For that reason we will still often call it the exponential map. 
\end{enumerate}
On the other hand,  our \emph{definition} of the Airy automorphic datum requires only the existence of a non-degenerate $\mathrm{Ad}$-invariant bilinear form. Such a form always exists if $G=\GL_n$ or if $G$ is almost simple and $p$ is a good prime for $G$, cf. \cite{JantzenLie}. 
\end{rem}

\subsection{Outline of our approach} 
The paper is structured as follows. In Section \ref{s:airydatum} we set up basic notations, explain our motivation for the choice of level group and define the Airy automorphic datum. We recall basic facts about the moduli space of $G$-bundles with level structure and prove that in our situation it is zero-dimensional. This is an indication of rigidity. 

Section \ref{s:rigidity} is the technical heart of the paper. In this section we analyze the moduli stack $\Bun_J$ in terms of double cosets using the Birkhoff decomposition 
\begin{equation} 
 \Bun_{J} =L^-G\backslash G(F)/J=\bigsqcup_{\mu\in X_*(T)}L^-G\backslash L^-G t^\mu I(1)/J. 
 \end{equation} 
Here $L^-G = L^+G \cap G(k[t,t^{-1}])$ for a uniformizer $t$ at $\infty$, $I(1)$ denotes the pro-unipotent radical of the standard Iwahori subgroup and $J$ is the group defined in Section \ref{s:airydatum}. We identify the part of $\Bun_{J}$ which is relevant in the sense that it can support automorphic sheaves. The main result of this section is the following. 
\begin{thm}
	The relevant orbits are $L^-Gt^\mu J$ where $\mu$ is any minuscule coweight.
\end{thm}
Since the connected components of $\Bun_J$ are labelled by minuscule coweights the above theorem says that there is a unique relevant point on each component of $\Bun_J$. The main ingredient in the proof of rigidity is a detailed analysis of the cosets $L^-G\backslash L^-Gt^\mu I(1)/J$  where $\mu$ is minuscule. We construct a subscheme $A_\mu$ of $I(1)$ such that for any minuscule coweight we have a decomposition
\[I(1)=(\Ad_{t^{-\mu}}U_\mu)A^\mu J\] 
which in particular implies that $L^-Gt^\mu I(1)=L^-G t^\mu A^\mu J$. Then to prove rigidity we need to show that if $L^-Gt^\mu a J$ with $a\in A^{\mu}$ is relevant then $a$ has to be the identity. This is done in Section \ref{s:proof of thm}. 

To construct the Hecke eigensheaf in this situation we want to adopt the argument originally used by Heinloth, Ngô and Yun. Its key ingredient is that the embedding of the trivial bundle into $\Bun_{J}$ is affine. For that reason in Section \ref{s:geomrelevantorb} we study the geometry of the relevant orbits. In our situation the relevant orbits are not open in $\Bun_{J}$, but we prove that they are closed in the generic locus, see Proposition \ref{p:geom of relevant orbits}. Since the generic locus is affine, we conclude that the relevant orbit is itself affine. 

In Section \ref{s:heckeeigen} we construct the Hecke eigensheaf and we extract the eigen local system $\mathrm{Ai}_{\hG}$ using the techniques established in \cite{HNY13} and  \cite{Yun14}. We explain how our argument fits into the general framework of \cite[Appendix A]{JY20}. Roughly speaking the tensor category $\Rep(\hG)$ acts factorizably via Hecke operators on the category $\Perv(\chi)$ of automorphic sheaves and from this action one extracts the eigenvalue. 

In Section \ref{s:Identification}, we compute the eigenvalue for $G=\GL_n$ and the standard representation. We do this by computing its Frobenius trace function and identifying it with the Frobenius trace of a classical Airy sheaf. For this computation we need to understand the character $\chi$ completely explicitly in terms of matrices, cf. Proposition \ref{p:mu}. 

Finally, in the appendix, we gather some facts about minuscule (co-)weights and their relationships to cyclic gradings and parabolic subgroups.

\subsection{Open problems} 
\subsubsection{Slope and Swan conductor} 
By analogy with the complex geometric situation we make precise conjectures about the local structure of the Airy sheaves we construct, e.g. about their slope and their Swan conductor; ss \S \ref{ss:Conjectures}. In the case of $\GL_{n}$, we confirm that the newly constructed local systems agree with classical Airy sheaves as defined by Katz, in particular confirming the conjectures in this case.

\subsubsection{Identification of the complex local system} 
Our construction is entirely geometric and may be carried out over the complex numbers as well (by replacing Artin-Schreier sheaves with a rank one connection of exponential type). In this way one obtains a family of connections on $\bA^{1}$. We conjecture that the Airy connections constructed in \cite{KS21} are special cases of those connections we construct geometrically. Identifying these connections is related to the work of X. Zhu who identifies Heinloth-Ngô-Yun's Kloosterman sheaves with the Frenkel-Gross connection in \cite{Zhu17, DaxinXinwen, KXY}.

Moreover, from the computation of eigenvalue local systems $\mathrm{Ai}_{\hG}$ for $G=\GL_n$, we can see that in this case the family $\mathrm{Ai}_{\hG}$ is strictly larger than the particular Airy connection defined in \eqref{eq:AiryConnection}, see Proposition \ref{p:Frob trace of cE_mu} and Remark \ref{rem:extension}. It would be of interest to find the generalization of Airy $\hG$-connection in \eqref{eq:AiryConnection} that corresponds to the family of eigenvalues $\mathrm{Ai}_{\hG}$ for general $\hG$.

\subsubsection{Global monodromy group} 
The global monodromy group of  $\GL_n$-Airy sheaves in characteristic $p>n$ was computed by Katz in \cite{Katz87}. The analogous group in characteristic zero; i.e., the differential Galois group of $\GL_n$-Airy connections, was also computed by Katz in \cite{KatzGroups}. The methods of \cite{KatzGroups} have recently been generalised to the reductive setting \cite{KS22}. In particular, the differential Galois group of Airy $G$-connections are computed in \emph{op. cit}. In analogy with the motivic picture of \cite{KatzHypergeometric}, we conjecture that for $p>h$ the global monodromy of the $\hG$-Airy sheaves coincide with the differential Galois group of the $\hG$-Airy connection. The global monodromy of $\GL_n$-Airy sheaves in small characteristic is a fascinating subject with links to finite simple groups, see for example \cite{KatzFinite}, \cite{KatzSuzuki}. As mentioned in Remark \ref{rem:characteristic} the automorphic datum for $\GL_n$ can be defined in arbitrary characteristic, but our proof crucially uses the assumption $p > h$. It would be interesting to construct $\hG$-Airy sheaves in small characteristic whenever it is possible to define the automorphic datum.

\subsubsection{Stokes data}
As mentioned above, the study of asymptotics of the classical Airy equation goes back to Stokes. Recently, the first author and Hohl gave an explicit description of the Stokes data of $\GL_n$ Airy connections using the Fourier-Laplace transform for $\cD$-modules \cite{JakobHohl}. Explicit determination of the Stokes data for $\hG$-Airy sheaves is an interesting open problem.

\subsection{Acknowledgement} We would like to thank Daniel Sage, Zhiwei Yun, and Xinwen Zhu for helpful conversations. KJ was supported by the DFG Research Fellowship JA 2967/1-1. MK was supported by Australian Research Council Discover Grants.

\section{The Airy automorphic datum} \label{s:airydatum}
Let's briefly outline our approach for constructing Airy sheaves for reductive groups. Let $J \subset L_\infty^{+} G$ be a pro-algebraic subgroup of finite codimension and denote by $\Bun_{J}$ the moduli stack of $G$-bundles on $X$ with $J$-level structure at $\infty$. One can think of this stack as classifying $G$-torsors $\cE$ on $X$ together with a trivialization $\iota_\infty : \cE | _{\Spec(\cO_\infty)} \cong G \times \Spec(\cO_\infty)$ up to left multiplication by $J$. 

The strategy can be summarized in the following steps. 
\begin{enumerate}
\item Construct a suitable level group $J$ and a character $\chi : J \to \bQlt$ factoring through a finite dimensional quotient of $J$ over $k$ to obtain a perverse sheaf $\cA_{\chi}$ on $\Bun_{J}$. 
\item Prove that there is a unique point on each connected component of $\Bun_{J}$ that can support a perverse sheaf satisfying equivariance properties with respect to $J$ and $\chi$, an indication of rigidity. 
\item Conclude that $\cA_{\chi}$ is a Hecke eigensheaf whose eigenvalue is an $\hG$-local system on $\bA^1$.
\end{enumerate}
We usually refer to the pair $(J,\chi)$ as an automorphic datum. This terminology was introduced in \cite{Yun14} where Yun develops a framework for extracting Hecke eigensheaves from rigid automorphic data. For details on the notion of automorphic data we refer to \S 2.3. and \S 2.6. of \emph{op. cit.}.

\subsection{Notation and setup}

\subsubsection{The curve and other data} 
Let $k$ be a finite field of characteristic $p$ and fix a prime $\ell \neq p$. Let $K$ be the field of rational functions on $X=\bP^1$. Fix an affine coordinate $t$ on $X\setminus \{0\}$ and identify $K=k(t)$. For a closed point $x\in |X|$ we denote by $\cO_x$ the completed local ring at $x$ with maximal ideal $\cP_x$ and field of fractions $K_x$. We abbreviate $F=K_\infty$, $\cO = \cO_\infty$, $\cP = \cP_{\infty}$ and identify $\cO = k[\![t]\!]$, $\cP = tk[\![t]\!]$. The ring of adèles of $K$ is the restricted product
\[\bA_K =\sideset{}{'} \prod_{x\in |X|} K_x. \]
We fix a nontrivial additive character $\psi_{k}:k \rightarrow \bQlt$ determining an Artin-Schreier sheaf $\cL_{\psi}$  on $\bA^1$.

\subsubsection{The group and its Lie algebra} 
Let $G$ be a connected split reductive group over $k$ with almost simple derived subgroup. We will always make the following assumption.
\begin{center} The characteristic $p$ of $k$ is greater than the Coxeter number $h$ of $G$. \end{center}

Let $B$ be a Borel subgroup of $G$ with maximal torus $T$ and unipotent radical $U$. Denote the opposite Borel and its unipotent radical by $B^-$ and by $U^-$. We consider the corresponding Lie algebras $\fg, \fb, \fb^-, \fu, \fu^-$ and $\ft$, which is a Cartan subalgebra. The assumption $p>h$ guarantees the existence of a non-degenerate $\Ad$-invariant symmetric bilinear form $\kappa:\fg \times \fg \rightarrow k$ and we will fix one such form. 

Let $(\Phi, X^*(T), \Phi^{\check{}}, X_*(T))$ be the root datum of $G$, $\Delta=\Delta(G)$ be the set of simple roots with respect to $B$, and let $W=N_G(T)/T$ be the Weyl group of $G$. We let $\widehat{G}$ be the reductive group over $\bQl$ whose root datum is dual to that of $G$. 
 
For each root $\alpha\in\Phi$ we have the root subspace $\fg_\alpha \subset \fg$ and the root subgroup $U_\alpha \subset G$. We choose a basis $E_\alpha$ of $\fg_\alpha$ for each $\alpha$ and define the principal nilpotent element $N^-:=\sum_{\alpha\in\Delta}E_{-\alpha} \in \fg$. 

We will denote the highest root of $\Phi$ by $\theta$ and let $\check{\rho}$ be the half sum of positive coroots. For a subspace $\fh\subset\fg$ that is normalized by $\ft$, we denote its $\ft$-weights by $\Phi(\fh)$.

\subsubsection{Loop groups} Denote by $LG$ the loop group of $G$. For a point $x \in X$ we will also denote by $L_x G$ the loop group at $x$. It is an ind-scheme over the residue field of $x$. The positive loop group $L^+_{x} G$ is a pro-algebraic group over the residue field. We will also use the notation $L^-G = G(k[t^{-1}])$.

\subsubsection{Cyclic grading}\label{sss:cyclic grading}
Since $p>h$ the element $\check{\rho}/h$ defines an order $h$ cyclic grading on $\fg$
\begin{equation}\label{eq:grading}
\fg=\bigoplus_{r\in\bZ/h\bZ}\fg_r
\end{equation}
that is related to the decomposition according to heights of roots in the following way. For any root $\alpha \in \Phi$ let $\Ht(\alpha):=\langle \alpha, \check{\rho}\rangle$ be its height. For $1\leq r\leq h-1$ or $-(h-1)\leq r\leq -1$ define the subspace
$$
\fg(r):=\bigoplus_{\Ht(\alpha)=r}\fg_\alpha
$$
With the notation $\fg(0)=\ft$ we then have
\begin{equation}
\fg_r=\fg(r)\oplus\fg(r-h), \quad \forall\ 1\leq r\leq h-1;\qquad\qquad \fg_0=\fg(0)=\ft.
\end{equation}

For later use we also introduce the following notation.
$$
\fu_{\geq r}:=\bigoplus_{\Ht(\alpha)\geq r}\fg_\alpha=\bigoplus_{r\leq i\leq h-1}\fg(i),\quad
\fu^-_{\leq -r}=\bigoplus_{\Ht(\alpha)\leq -r}\fg_\alpha=\bigoplus_{1-h\leq i\leq -r}\fg(i),\quad
\forall\ 1\leq r\leq h-1.
$$

\subsubsection{The Coxeter torus} Since $p>h$ it does not divide the order of the Weyl group of $G$. Hence there is a well-known bijection between $G(F)$-conjugacy classes of maximal $F$-tori and conjugacy classes in the Weyl group $W$ (cf. \cite[\S 4.2]{GKM06}). We will refer to the torus corresponding to the Coxeter class as \textit{Coxeter torus}.

\subsection{Local structure of Airy sheaves at $\infty$} \label{sss:localAiry} Before we define $(J,\mu)$ in the next section, we give some motivation explaining our choice of level group $J$ and character $\mu$. This is informed mostly by the local Langlands correspondence, in particular by the construction of tame supercuspidal representations due to Adler \cite{Ad98}, Yu \cite{Yu01} and Fintzen \cite{Fin21}. 

To start with, assume that $G=\GL_n$. In this case Katz has constructed Airy sheaves using Fourier transform in \cite{Katz87} as follows.  
Let $\pr_1,\pr_2:\bA^2\ra\bA^1$ be the projections to the first and second coordinate, $m:\bA^2\ra\bA^1$ be the multiplication map.  For a constructible sheaf $\cF$ on $\bA^1$ its naive Fourier transform is defined as
$$
\mathrm{NFT}_\psi(\cF)=R^1\pr_{2,!}(\pr_1^*\cF\otimes m^*\cL_\psi).
$$

Recall that $\cL_{\psi}$ denotes the Artin-Schreier sheaf associated to the character $\psi : k \to \bQl$. Assume that $p > n+1$. Given a polynomial $f(x)\in k[x]$ of degree $n+1$, $n\geq 1$ we consider the rank one local system $\cL_{\psi(f)}=f^*\cL_\psi$ on $\bA^1$. Airy sheaves are by definition Fourier transforms of sheaves of the form $\cL_{\psi(f)}$, that is 
\begin{equation}
\cF_f:=\mathrm{NFT}_\psi(\cL_{\psi}(f))
\end{equation}
is an Airy sheaf in the sense of Katz. By \cite[Theorem 17]{Katz87} the sheaf $\cF_f$ is lisse on $\bA^1$ of rank $n$, and all of its breaks at $\infty$ are equal to $\frac{n+1}{n}$.

We can say even more about the local structure of an Airy sheaf using Fu's calculation of $\ell$-adic local Fourier transforms in \cite{Fu10}. Let $D_{\infty}^{\times}$ be a formal puncture disc around $\infty$. By the stationary phase formula 
\[
\cF_f |_{D_{\infty}^{\times}} \cong \mathrm{FT}^{(\infty, \infty)} (\cL_{\psi}(f) |_{D_{\infty}^{\times}} ) 
\]
where $\mathrm{FT}^{(\infty, \infty)}$ denotes Laumon's local Fourier transform. Let $E$ be the totally ramified tame extension $E$ of $F$ of degree $n$. By \cite[Theorem 0.1.]{Fu10}, the right hand side is isomorphic to a sheaf on $D_{\infty}^{\times}$ corresponding to the representation $\Ind_{E|F}(\xi)$ for some suitable character $\xi$ of level $n+1$ of $\mathrm{Gal}(F^{\mathrm{sep}} / E)$ (resp. of $E^{\times}$ via local class field theory).

\subsubsection{Recollections on local Langlands} 
Let us assume $G=\GL_n$ for simplicity. Because of our assumption $p>h=n$ we are in the tame case of the local Langlands correspondence which is easier to describe. Since the degree of $f$ is $n+1$ it is never divisible by $n$ and the pair $(E / F, \xi)$ is admissible in the sense of Bushnell and Henniart \cite[\S 3]{BH}. 
In this case the local Langlands correspondence is described rather explicitly. Roughly speaking the supercuspidal representation corresponding to $(E / F, \xi)$ is of the following form. Let $I\subset \GL_{n}(F)$ be the standard Iwahori subgroup with Moy--Prasad subgroups $I(r),r\in\bN$. For even $n$ the supercuspidal representation is of the form 
\begin{equation} \label{eq:even} 
c-\Ind_{E^\times I(\frac{n}{2}+1) }^{G(F)} \Lambda
\end{equation} 
for a certain character $\Lambda$ determined by $\xi$. If $n$ is odd, the situation is more complicated. The supercuspidal representation is of the form
\begin{equation} \label{eq:odd} 
c-\Ind_{E^\times I(\frac{n+1}{2}) }^{G(F)} \Lambda',
\end{equation} 
but now $\Lambda'$ is a certain (uniquely defined) irreducible representation of $E^\times I(\frac{n+1}{2})$. Write $n+1=2m$, let $\fp_E$ be the maximal ideal of $E$, let $U_E^1=1+\fp_E$ and write $V=U_E^1I(m)/U_E^1I(m+1)$. The character $\xi$ determines a character on $I(n+1)$ which extends to a character $\theta$ on $U_E^1(m+1)$. Composing the commutator pairing on $I(m)$ with $\theta$ defines a symplectic form
\[
 V \times V \rightarrow \mathbb{C}^\times. 
\]
Attached to the symplectic space $V$ is a Heisenberg group and its Heisenberg representation with central character $\theta$ can be used to determine the representation $\Lambda'$. On the other hand, to construct an automorphic datum we need to work with a character, not a representation. For that we want to extend the character $\theta$ on $U_E^1(m+1)$ as much as possible. The necessary condition to extend $\theta$ further is encoded in the symplectic form: $\theta$ extends to any isotropic subspace of  $V$. That means to extend $\theta$ as far as possible we should choose a maximal isotropic subspace of $V$, that is a Lagrangian subspace $L \subset V$. Note that in any case $E^\times$ is the Coxeter torus. 

\subsubsection{Back to the case of general reductive group} Now let us consider the case of general reductive $G$. Then the key point is that Adler, Yu and Fintzen  \cite{Ad98, Yu01, Fin21} have proved that (appropriate modifications of the) representations defined in \eqref{eq:even} and \eqref{eq:odd} continue to be irreducible (and supercuspidal). One of the key inputs in their construction is a sequence of twisted Levi subgroups. In our case this is only the Coxeter torus. The preceding discussion suggests the choice of level group in the following section. 

\subsection{The level group and the character}
Let $\kappa$ be the non-degenerate $\Ad$-invariant symmetric bilinear form we fixed before. For $\fg(F)=\fg\otimes_k F$ its composition with the residue gives a perfect pairing
\begin{equation}\label{eq:pairing}
\fg(F)\times\fg(F)\ra k: (X,Y)\mapsto 
\Res\left(\kappa(X,Y)\frac{dt}{t}\right)
\end{equation}

Consider $\tilde{X}:=t^{-1}N^-+t^{-2}E_\theta\in\fg(F)$. Its dual with respect to the above pairing defines a $k$-linear form
$$
\tilde{\phi}=\tilde{X}^*:Y\mapsto \Res\left(\kappa(\tilde{X},Y)\frac{dt}{t}\right).
$$
\begin{rem} The element $\tilde{X}$ is \textit{good} in the sense of \cite[Definition 2.2.4]{Ad98} and its centralizer is the Coxeter torus. In this way our definitions fit into Adler's construction of supercuspidal representations. 
\end{rem}
As explained above we need to define the level group in two different ways depending on the parity of $h$. 
\subsubsection{Even Coxeter number}
Assume  that the Coxeter number $h$ of $G$ is even. This includes all types of $G$ except type $A_{2n}$. We let $I\subset G(F)$ be the standard Iwahori subgroup with Moy--Prasad subgroups $I(r),r\in\bN$. Consider the restriction $\phi:=\tilde{\phi}|_{\Lie(I(1+h/2))}\in\Lie(I(1+h/2))^*$. For any $r$ the group $I(r)$ acts on $\Lie(I(1+h/2))^*$. Note that $\tilde{X}$ is a regular semisimple element and its centralizer $S:=Z_{G(F)}(\tilde{X})$ is the Coxeter torus. Let $S(r):=S\cap I(r)$\footnote{Note that this is smaller than $\Stab_{I(r)}(\phi)$, because the latter one contains $I(1+h/2)$ when $r\leq 1+h/2$.} and note that it stabilizes $\phi$. From the definition of $\phi$, we can see that it is trivial on $\Lie(I(2+h))$. Since $I(1+h/2)/I(2+h)$ is a unipotent commutative quotient, the exponential map gives an isomorphism $\Lie(I(1+h/2))/\Lie(I(2+h))\simeq I(1+h/2)/I(2+h)$. Thus $\phi$ also defines a character on $I(1+h/2)$, which is trivial on $I(2+h)$. 

Recall that we fixed a non-trivial additive character $\psi:k \rightarrow \bQlt$. Then $\psi\circ\phi:I(1+h/2)\ra\bQlt$ gives a character which is stabilized by $S(1)$ from definition of $\phi$. Since $S(1)$ is abelian, we can extend $\psi\phi$ to $J:=S(1)I(1+h/2)$, which we denote by $\chi$. 

\subsubsection{Odd Coxeter number}
For $G$ of type $A_{2n}$ we consider the parahoric subgroup $P\supset I$ defined as the stabilizer of the point $\check{\rho}/2n$ in the Bruhat-Tits building of $G(F)$. Observe that
\begin{equation}\label{eq:P(1+n)}
P(2)\subset I(2),\ I(n+2)\subset P(1+n)\subset I(n+1);\ P(1+n)/I(n+2)\simeq\fg(n+1),\ I(n+1)/P(1+n)\simeq\fg(n+1-h).
\end{equation}
Similar as before, we let $J=S(1)P(1+n)$. Again $\phi=\tilde{\phi}|_{\Lie(P(1+n))}$ is trivial on $\Lie(I(2+h))$. Since $[P(1+n),P(1+n)]\subset P(2+2n)\subset I(2+h)$, we see $P(1+n)/I(2+h)$ is commutative. Thus $\phi$ is a character of Lie algebra, which induces a character $\phi$ on $P(1+n)$. Let $\chi$ be any extension of $\psi\phi$ from $P(1+n)$ to $J=S(1)P(1+n)$.
\begin{defe}\label{d:Airy}
	We call $(J,\chi)$ the Airy automorphic datum.
\end{defe}

\subsection{The moduli stack of $G$-bundles with level structures}
We recall a few facts about the structure of the moduli stack of $G$-bundles. We define $J^{+} := I(h+2)$ and let $\Bun_{J}$ be the moduli stack of $G$-bundles on $X$ with level structure $J$ at $\infty$ and $\Bun_{J^{+}}$ be the moduli stack of $G$-bundles with $J^{+}$-level structure at $\infty$. Note that $J^{+}$ is a normal subgroup of $J$ and that $\Bun_{J^{+}}$ carries an action of $L:=J/J^{+}$ by changing the level structure. The natural map $\Bun_{J^{+}}\rightarrow \Bun_{J}$ is an $L$-torsor. 

\subsubsection{Birkhoff decomposition}\label{sss:birk}
Denoting $L^{-}G =  G (k[t^{-1}])$ we can use one-point uniformization to identify
$$
\BunJ = L^{-}G \backslash LG / J 
$$ 
by \cite[\S 2.12.]{Yun16}. Using the Birkhoff decomposition 
$$
LG = \bigsqcup_{\mu\in X_*(T)} L^-G t^\mu I(1)
$$ 
we obtain a stratification of $\BunJ$ into locally closed substacks indexed by coweights of $G$.

\subsubsection{Kottwitz map and connected components}\label{sss:comp}
The Kottwitz map $\kappa: \Bunpl \rightarrow X^*(Z\hG)$ induces an isomorphism 
\[\pi_0(\Bunpl) \cong X^*(Z\hG) \cong X_*(T) / \Lambda^{\vee}\]
where $\Lambda^{\vee}$ is the coroot lattice. For each $\alpha \in X^*(Z\hG)$ we denote the corresponding component by $\Bunpl^{\alpha}$. More explicitly, given a double coset represented by $g \in G(F)$, by the Cartan decomposition there is a unique coweight $\mu$ such that $g\in G(\cO) t^\mu G(\cO)$ and $\kappa$ maps the given double coset to the image of $\mu$ under the projection map $X_*(T) \rightarrow X^*(Z\hG)$. 

\subsection{Identification of $\pi_0(\Bunpl)$ with minuscule coweights} \label{ss:components}
The following lemma is standard: 
\begin{lem} Let $M^{\vee} := \{ \mu \in X^{*}(T) \mid 0 \le \langle \alpha^{\vee}, \mu \rangle \le 1\, \forall \alpha \in \Phi^+\}$ be the set of miniscule weights and $\Lambda\subset X^*(T)$ the root lattice. Then the canonical projection $M^{\vee} \to X^*(ZG)=X^{*}(T) / \Lambda$ induces a bijection (of sets). 
\end{lem}
\begin{proof}  If $G$ is almost simple and simply connected, $M^{\vee}$ is the set of minuscule weights (including $0$) and this is proven in \cite[Chapter VI, \S 2.3.]{Bourbaki}. For a not necessarily simply connected almost simple group $G$ we may consider its simply connected cover $G^{\mathrm{sc}}$ with central isogeny $G^{\mathrm{sc}} \to G$. Even more generally when $G$ only has almost simple derived group, there is a central torus $R(G) \subset G$ such that the product map $R(G) \times G^{\der} \rightarrow G$ is a central isogeny. Once we know that the claim holds for almost simple groups, it is easy to verify for $R(G) \times G^{\der}$. Thus it remains to be shown that given a central isogeny $\phi: G\to G'$ of reductive groups such that the statement holds for $G$ it also holds for $G'$. The isogeny $\phi$ induces an injective map $X^*(T') \to X^*(T)$ that is compatible with the canonical pairing $\langle -,- \rangle$. Hence we obtain a commutative diagram (of sets)
\[\xymatrix{M'^{\vee} \ar[r]^{\phi} \ar[d]& M^{\vee} \ar[d] \\
X^*(T')/\Lambda' \ar[r]^{\phi} & X^*(T)/\Lambda
}\]
in which all maps are injective and $M^{\vee} \rightarrow X^*(T)$ is surjective in addition. Given $\mu'\in X^*(T')$ we let $\mu \in M^{\vee}$ such that $\mu+\Lambda = \phi(\mu')+\Lambda$. Then there are $\lambda_1, \lambda_2 \in \Lambda$ such that $\mu = \phi(\mu')+\lambda_2-\lambda_1$. Since $\phi|_{\Lambda'}: \Lambda' \rightarrow \Lambda$ is surjective we furthermore can find $\lambda'_1, \lambda'_2$ such that $\mu = \phi(\mu'+\lambda'_2 - \lambda'_1)$. Since $\mu \in M^{\vee}$ we have $\mu'+\lambda'_2 - \lambda'_1\in M'^{\vee}$ and the image in $X^*(T')/\Lambda'$ of this is the class of $\mu'$. This proves that $M'^{\vee}\to X^*(T')/\Lambda$ is surjective. 
\end{proof}

Let $M := \{ \mu \in X_*(T) \mid 0 \le \langle \alpha, \mu \rangle \le 1\, \forall \alpha \in \Phi^+\}$ be the set of miniscule coweights. 

\begin{cor} $\pi_0(\Bunpl) = X^*(Z\hG)=X_*(T) / \Lambda^{\vee}= M$
\end{cor} 

\subsubsection{} Let $\widetilde{W} := N_{(G(F)}(T(F) / T(\cO))$ be the Iwahori-Weyl group and let $\Omega$ be the stabilizer in $\widetilde{W}$ of the alcove defining the standard Iwahori subgroup from above. Then $\widetilde{W} \cong W_{\aff} \ltimes \Omega$ where $W_{\aff}$ is the affine Weyl group of $G(F)$ and furthermore $\Omega \cong X^*(Z\hG)$.

Since $J^{+}=I(h+2)$ is a Moy--Prasad subgroup of the standard Iwahori subgroup, $\Omega$ acts on $\Bunpl$ by changing the level structure. For $\alpha \in \Omega$ this allows us to identify the connected components of $\Bunpl$ via isomorphisms
\begin{align}\label{eq:infHecke}
\bfT_\alpha : \Bunpl^{0} \simeq \Bunpl^{\alpha},
\end{align}
see also \cite[\S 4.3.3.]{KY20}.

\subsubsection{Dimension of $\Bun_{J}$} 
The first evidence for rigidity of the Airy automorphic datum is a computation of the dimension of $\Bun_{J}$. For the sake of notation let us assume that $G$ is simple in this subsection. 
\begin{prop} The moduli stack $\Bun_{J}$ is zero-dimensional. 
\end{prop}
\begin{proof} By \cite[\S 2.7.11]{Yun14} we have $\dim \Bun_{J} = -\dim(G) + \dim G(\cO_{\infty}) / J.$ First assume that $h$ is even. Consider the principal grading $\fg=\bigoplus_{r \in\bZ/h\bZ}\fg_r$, let $\Phi^{+}$ be the set of positive roots and let $k_r$ denote the number of positive roots of height $r$. Then $\dim \fg_r = k_r + k_{h-r}$ and
$$
\dim (I(1)/I(1+h/2)) = \sum_{r=1}^{h/2} \fg_{r} = \#\Phi^+ + k_{h/2}. 
$$
This implies that
$$
 \dim G(\cO_{\infty}) / I(1+h/2) = \dim B + \dim I(1) / I(1+h/2) = \dim G + k_{h/2}.
 $$
Therefore it remains to prove that $\dim S(1)/S(1+h/2) = k_{h/2}$. From the proof of \cite{Panyushev}[Theorem 4.2.(i)] where we take the grading defined by $\check{\rho}/h$, we see the dimension of $S(1)/S(1+h/2)$ equals the number of exponents of $\fg$ that are $\leq h/2$. From Theorem 4.2.(i) of $loc. cit.$, we see this number equals the number of exponents of $\fg$ that are $\geq h/2$, thus equals half of the sum of $\mathrm{rk}\fg$ with the multiplicity of $h/2$ in exponents. It follows from Theorem 4.2.(ii) of $loc. cit.$ that this equals $k_{h/2}$. 

The only case where $h=2n+1$ is odd is in type $A_{2n}$. In that case $J=S(1)P(1+n)$ and the parahoric subgroup $P$ is defined by the barycenter $\check{\rho}/2n$. One easily checks that 
$$
\dim G(\cO) / P(1+n) = \dim(G)+(n+1).
$$
Furthermore $S(1) \cap P(1+n) = S(n+2)$ and $\dim S(1) / S(n+2)$ is equal to the number of exponents of $\fg$ that are $< n+2$ and this is precisely $n+1$. Thus $\dim G(\cO) / J = \dim G$. 
\end{proof}

\subsubsection{Relevant orbits} 
Recall that $L=J/J^{+}$ and that the natural map $p:\Bunpl \to \Bun_{J}$ is an $L$-torsor. For any bundle $\cE^{+} \in \Bunpl$ denote by $S_{\cE^{+}}$ the stabilizer of $\cE$ in $L$.

\begin{defe} Let $\cE^{+}\in \Bunpl$ be a $G$-bundle. We say that it is relevant if the restriction of the character $\chi$ to $S_{\cE^{+}}^{\circ}$ is trivial. Otherwise it is irrelevant. \end{defe}

A bundle $\cE\in \Bun_J$ is relevant if and only if some bundle in the fiber of $p$ over $\cE$ is relevant. The choice of any such preimage $\cE^{+}$ identifies the automorphism group $\Aut(\cE)$ of $\cE$ with $S_{\cE^{+}}$, cf \cite[\S 2.7.4]{Yun14}. In terms of the one-point uniformisation in \S\ref{sss:birk}, if the bundle $\cE$ corresponds to the double coset $L^{-}G t^{\mu} g J$ with $g\in I(1)$ and $\mu$ some coweight, then $\Aut(\cE) \cong \Ad_{g^{-1}t^{-\mu}} L^{-}G \cap J$ and it is relevant if and only if $\chi$ is trivial on this intersection.

\section{Rigidity of the Airy automorphic datum}\label{s:rigidity} 

\subsection{Parametrization of orbits}
It is our goal to prove that the automorphic datum $(J,\chi)$ is rigid in the sense that for every $\alpha\in \pi_0(\Bunpl)$ there exists a unique relevant orbit $O_{+}^{\alpha}$ on the component $\Bunpl^{\alpha}$ of $\Bunpl$. For that we need to parametrize the orbits of $\Bun_{J}$. The Birkhoff decomposition implies that

\begin{equation}\label{eq:orbits}
\Bun_{J} =L^-G\backslash G(F)/J=\bigsqcup_{\mu\in X_*(T)}L^-G\backslash L^-G t^\mu I(1)/J.
\end{equation}

Recall that the orbit $L^-G t^\mu a J$, $a\in I(1)$ is relevant if and only if 
\begin{equation}\label{eq:relevant}
\chi(\Ad_{a^{-1}t^{-\mu}}L^-G\cap J)=1.
\end{equation}

As explained in \S\ref{sss:comp} the connected components of $\Bun_{J}$ can be parametrized by minuscule coweights. The main statement now is the following.
\begin{thm}\label{t:rigidity}
	The relevant orbits are $L^-Gt^\mu J$ where $\mu$ is any minuscule coweight.
\end{thm}

We will see that it is easy to show there are no relevant orbits in $L^-G t^\mu I(1)$ unless $\mu$ is minuscule. To work with the minuscule cosets, we will need a finer parametrization for the minuscule cosets $L^-G\backslash L^-Gt^\mu I(1)/J$. We first establish this parametrization and use it to prove Theorem \ref{t:rigidity} in Section \ref{s:proof of thm}. The parametrization involves a factorization of $I(1)$ into a product of three groups, one of which is $J$.

\subsection{A finer parametrization of minuscule cosets} \label{ss:fine param}
Recall that $N^-=\sum_{\alpha\in\Delta}E_{-\alpha}$. There exists a unique regular nilpotent element $N^+\in\fu$ such that $\{N^-,2\check{\rho},N^+\}$ is a principal $\mathfrak{sl}_2$-triple. Consider the following subspaces of $\fg_r$:
\begin{equation}\label{eq:a_r}
\fa_r:=[\fg(r-1-h),N^+]\subset \fg(r-h)\subset\fg_r, \quad 1\leq r\leq h-1.
\end{equation}

Note that $\fg(-h)=0$, thus $\fa_1=0$. Define $\fa:=\bigoplus_{r=1}^{\lceil h/2\rceil}\fa_r\subset\fu_{\leq -\lfloor h/2\rfloor}^-$. Observe that $\fa t\subset\Lie(I(1))$. Thus $A:=\exp(\fa t)\subset I(1)$ is a $k$-subscheme. 

Similarly, for any minuscule coweight $\mu$ we have defined a nilpotent subalgebra $\fu_\mu^-$, a Weyl group element $w_P$ and a unipotent subgroup $U_\mu$ in \S\ref{a ss:parabolic}. Denote $\fa_r^\mu:=\Ad_{w_P}\fa_r$. From Lemma \ref{l:u conju u_mu} and Lemma \ref{l:w_P0w_0}.(3), $\fa_r^\mu\subset\fg_r\cap\fu_\mu^-=:\fu_r'^-$. Let $\fa^\mu:=\Ad_{w_P}\fa\subset\fu_\mu^-$. Note that $t\Ad_{t^{-\mu}}\fu_\mu^-=t\fu_M^-\oplus\fu_P$, thus $A^\mu:=\exp(t\Ad_{t^{-\mu}}\fa^\mu)=\Ad_{t^{-\mu}w_P}A\subset I(1)$. We have the following refined parametrization:

\begin{prop}\label{p:minuscule cosets}
	For any minuscule coweight $\mu$ we have  
	\[I(1)=(\Ad_{t^{-\mu}}U_\mu)A^\mu J\] and consequently $L^-Gt^\mu I(1)=L^-G t^\mu A^\mu J$.
\end{prop}

To prove the above Proposition and Theorem \ref{t:rigidity} we establish a series of preliminary results. Some of them are presumably known in the literature.

\subsubsection{}
Consider the projection $p_-:\fg_r=\fg(r)\oplus\fg(r-h)\ra\fg(r-h)$, $1\leq r\leq h-1$. Let $X_{-1}:=N^-+E_\theta\in\fg_{-1}$. As $X_{-1}$ is regular semisimple, $\fz:=\fz_\fg(X_{-1})$ is a Cartan subalgebra. It has a decomposition $\fz=\bigoplus_{r\in\bZ/h\bZ}\fz_r$ into eigenspaces of the grading \eqref{eq:grading}.

\begin{lem}\label{l:p_-(z_r)}
	The restriction of $p_-$ to $\fz_r$ induces a bijection $p_-:\fz_r \ra \ker(\ad_{N^-}|_{\fg(r-h)})$, $\forall\ 1\leq r\leq h-1$.
\end{lem}
\begin{proof}
Given any $Y\in\fz_r\subset\fg_r=\fg(r)\oplus\fg(r-h)$, assume $Y=Y_r+Y_{r-h}$, $Y_r\in\fg(r)$, $Y_{r-h}\in\fg(r-h)$. $p_-(Y)=Y_{r-h}$. It satisfies
$$
0=[Y,X_{-1}]=[Y_r+Y_{r-h},N^-+E_\theta]=[Y_{r-h},N^-]+([Y_r,N^-]+[Y_{r-h},E_\theta])\in\fg(r-1-h)\oplus\fg(r-1).
$$

Thus $[Y_{r-h},N^-]=0$, $p_-(Y)=Y_{r-h}\in\ker(\ad_{N^-}|_{\fg(r-h)})$. Now we get a map $p_-:\fz_r\ra\ker(\ad_{N^-}|_{\fg(r-h)})$. The dimension of source and target are both equal to $k_r:=\#\{e_j|e_j=r\}$, i.e. the number of exponents that equal $r$. In fact, $\dim\fz_r=k_r$ is explained in the proof of \cite[Theorem 4.2.(i)]{Panyushev}, where it is also shown that $k_r=k_{h-r}$. The equality $\dim\ker(\ad_{N^-}|_{\fg(r-h)})=k_r=k_{h-r}$ is clear from the decomposition of $\fg$ into weight subspaces with respect to the principal $\mathfrak{sl}_2$-triple $\{N^-,2\check{\rho},N^+\}$. Therefore to show that $p_-$ is bijective it suffices to show it is injective. 

Given a basis $Z_1,Z_2,...,Z_{k_r}$ of $\fz_r$, where $Z_i=Y_{r-h}^i+Y_r^i$ for $Y_{r-h}^i\in\fg(r-h)$, $Y_r^i\in\fg(r)$. Suppose
$$
p_-(\sum_{i=1}^{k_r}\lambda_i Z_i)=\sum_{i=1}^{k_r}\lambda_i Y_{r-h}^i=0
$$

Then $\sum_{i=1}^{k_r}\lambda_i Z_i=\sum_{i=1}^{k_r}\lambda_i Y_r^i\in\fg(r)\cap\fz_r$ is both nilpotent and semisimple and thus equals zero. Therefore, $p_-$ is injective and furthermore bijective.
\end{proof}

\begin{lem}\label{l:g_r decomp} Denote $\fu_{\mu,r}=\fu_\mu\cap\fg_r$, $\fu_{\mu,r}^-=\fu_\mu^-\cap\fg_r$.
For any minuscule coweight $\mu$ and all $1\leq r\leq h-1$ we have 
\[\fg_r=\fz_r\oplus\fa_r^\mu\oplus\fu_{\mu,r}.\]
\end{lem}
\begin{proof}
First consider $\mu=1$. We show that $\fg_r=\fz_r+\fa_r+\fg(r)$. To this end, it suffices to show that $\fg(r-h)=p_-(\fz_r)+\fa_r$. From Lemma \ref{l:p_-(z_r)} and definition of $\fa_r$, this is same as
$$
\fg(r-h)=\ker(\ad_{N^-}|_{\fg(r-h)})+[\fg(r-1-h),N^+]
$$
which is clear from the $\mathfrak{sl}_2$-triple decomposition. Moreover it is clear that the sum above is indeed direct. 

Now we prove that $\fg_r=\fz_r+\fa_r+\fg(r)$ is actually a direct sum decomposition. Let $Y=Y_{r-h}+Y_r\in\fz_r$, $X\in\fa_r\subset\fg(r-h)$ and $Z\in\fg(r)$ such that $Y+X+Z=(Y_{r-h}+X)+(Y_r+Z)=0$. Then it follows that $Y_{r-h}+X=0$ and $Y_r+Z=0$. From the above discussion we see that $Y_{r-h}=X=0$. Thus $p_-(Y)=Y_{r-h}=0$. From Lemma \ref{l:p_-(z_r)}, we know that $p_-$ is bijective and thus $Y=0$, $Y_r=0$, $Z=0$. This shows the decomposition in the statement is a direct sum.

Next, for general $\mu$, we know from Lemma \ref{l:w_P0w_0}.(3) that $w_P$ acts on $\fg_r$. Also, from Lemma \ref{l:w_P0w_0}.(2), we know $w_P\in Z_G(X_{-1})$. Thus it acts trivially on $\fz_r$. Since $\Ad_{w_P}\fg(r)=(\Ad_{w_P}\fu)\cap(\Ad_{w_P}\fg_r)=\fu_{\mu,r}$, the conjugation by $w_P$ of the decomposition $\fg_r=\fz_r\oplus\fa_r\oplus\fg(r)$ provides the decomposition for $\mu$.
\end{proof}

\begin{cor}\label{c:factorization of I(r)}
	For any minuscule coweight $\mu$ and any $1\leq r\leq h-1$, every $g_r\in I(r)$ can be written as a product $g_r=u_ra_rs_rg_{r+1}$, where $u_r\in (\Ad_{t^{-\mu}}U_\mu)\cap I(r)$, $a_r\in\exp(t\Ad_{t^{-\mu}}\fa_r^\mu)$, $s_r\in S(r)$, $g_{r+1}\in I(r+1)$.
\end{cor}
\begin{proof}
We have two embeddings
\begin{align*}
S(r)/S(r+1)\simeq\fz_r\hookrightarrow\fg_r\simeq I(r)/I(r+1),\\
(\Ad_{t^{-\mu}}U_\mu)\cap I(r)/(\Ad_{t^{-\mu}}U_\mu)\cap I(r+1)\simeq\fu_{\mu,r}\hookrightarrow\fg_r\simeq I(r)/I(r+1).
\end{align*}

Given any $g_r\in I(r)$, consider its image $\bar{g}_{r}\in I(r)/I(r+1)\simeq\fg_r$. By Lemma \ref{l:g_r decomp}, we can write it as 
$$
\bar{g}_{r}=z+a+x,\quad z\in\fz_r,\ a\in\fa_r^\mu,\ x\in\fu_{\mu,r}.
$$

Let $s_r\in S(r)$ be the preimage of $z$ under the first embedding above, $u_r=\exp(\Ad_{t^{-\mu}}x)$ be the preimage of $x$ under second embedding above, $a_r=\exp(t\Ad_{t^{-\mu}}a)$, $g_{r+1}=(u_r a_r s_r)^{-1}g_r$. Then $\bar{g}_{r+1}=1\in I(r)/I(r+1)$, $g_{r+1}\in I(r+1)$ as desired.
\end{proof}

To conclude the proof of Proposition \ref{p:minuscule cosets} we need one more simple observation.
\begin{lem}\label{l:exp product}
	For any $Y_1\in\fa\subset\fu^-$, $Y_2\in\fa_r$ and any $1\leq r\leq h-1$ we have 
\[
\exp(Y_1 t)\exp(Y_2 t)=\exp((Y_1+Y_2)t)g
\] 
for some $g\in I(1+r)$.
\end{lem}
\begin{proof}
Note that $\fu^- t\subset\Lie(I(1))$. We want to show that
$$
\exp(-Y_1 t-Y_2 t)\exp(Y_1 t)\exp(Y_2 t)\in I(1+r)
$$

We know that for two elements $X,Y\in\fu^-$, $\exp(X)\exp(Y)=\exp(X+Y+[X,Y]/2+(\cdots))$, where $(\cdots)$ is a finite sum of iterated Lie brackets between $X$ and $Y$. Thus the above product is of the form $\exp(Z t)$, where $Z$ is a finite sum of iterated Lie brackets between $Y_1, Y_2$. This means $Zt\in\Lie(I(1+r))$, which completes the proof.
\end{proof}

\subsection{Proof of Proposition \ref{p:minuscule cosets}}\label{ss:proof of prop}
Now we are ready to establish the description of minuscule cosets. 
\begin{proof}
The second decomposition follows immediately from the first one and Lemma \ref{l:stab=u_mu}, thus it suffices to prove the first. Assume $h$ is even. Equivalently, we need to show that any $g\in I(1)$ can be written into a product $g=uasg'$, where $u\in \Ad_{t^{-\mu}}U_\mu$, $a\in A^\mu$, $s\in S(1)$, $g'\in I(1+h/2)$. This follows from Corollary \ref{c:factorization of I(r)}, Lemma \ref{l:exp product} and induction. More specifically:

First apply Corollary \ref{c:factorization of I(r)} to $r=1$ and $g_1=g$, we obtain $g=u_1a_1s_1g_2$, $a_1=\exp(v_1)$ for $v_1\in t\Ad_{t^{-\mu}}\fa_1^\mu$.

Next, apply Corollary \ref{c:factorization of I(r)} to $r=2$ and $g_2$, we obtain $g_2=u_2a_2s_2g_3$, $a_2=\exp(v_2)$ for $v_2\in t\Ad_{t^{-\mu}}\fa_2^\mu$. Thus
$$
g=u_1a_1s_1u_2a_2s_2g_3=(u_1u_2)(a_1a_2)(s_1s_2)g_3'
$$

Here we use that commutators of $I(1)$ and $I(r)$ are in $I(r+1)$. By Lemma \ref{l:exp product}, we have $a_1a_2=\exp(v_1+v_2)g_3''$ for some $g_3''\in I(3)$. Combining these, we get
$$
g=(u_1u_2)\exp(v_1+v_2)(s_1s_2)g_3'''
$$

Apply Corollary \ref{c:factorization of I(r)} to $r=3$ and $g_3'''$, we can similarly obtain $g=(u_1u_2u_3)\exp(v_1+v_2+v3)(s_1s_2s_3)g_4$. By induction, we get
$$
g=(\prod_{r=1}^{h/2}u_r)\exp(\sum_{r=1}^{h/2}v_r )(\prod_{r=1}^{h/2}s_r)g_{1+h/2}=uasg'
$$
where $u_r\in \Ad_{t^{-\mu}}U_\mu\cap I(r)$, $v_r\in t\Ad_{t^{-\mu}}\fa_r^\mu$, $s_r\in S(r)$. This proves the proposition.

For $h$ odd, recall from \eqref{eq:P(1+n)} that $I(n+2)\subset P(1+n)\subset I(n+1)$. Replacing $h/2$ with $\lceil h/2\rceil=n+1$ in above proof, we obtain $I(1)=(\Ad_{t^{-\mu}}U_\mu)A^\mu S(1)I(n+2)=(\Ad_{t^{-\mu}}U_\mu)A^\mu S(1)P(1+n)$ as desired.
\end{proof}

\section{Proof of Theorem \ref{t:rigidity}}\label{s:proof of thm}
The proof of our main rigidity result has four parts. 
\begin{enumerate}[label=(\roman*)]
\item We first prove that a coset $L^-G t^\mu aJ$ can only be relevant if $\mu$ is minuscule.
\item  Then we study the coset corresponding to the trivial coweight. If the coset $L^-GaJ$ with $a\in I(1)$ is relevant, then $a\in J$. 
\item We reduce the case of a general minuscule coweight to the case of the trivial one. 
\item Finally we prove that the cosets $L^-G t^\mu J$ are actually relevant. 
\end{enumerate}

In the following we prove for $h$ even. The proof for $h=2n+1$ odd is word-by-word the same as the even case, with $h/2$ replaced by $\lceil h/2\rceil=n+1$.

\subsection{Part (i): minuscule coweights}
Suppose $L^-G t^\mu aJ$ is relevant, $a\in I(1)$. Note that
$$
\Ad_{a^{-1}}(\Ad_{t^{-\mu}}L^-G\cap I(1+h))\subset\Ad_{a^{-1}t^{-\mu}}L^-G\cap J.
$$

If $\langle \alpha,\mu\rangle <0$ for some simple root $\alpha\in\Delta$, then 
$$
U_\alpha(kt)\subset\Ad_{t^{-\mu}}U_\alpha(k[t^{-1}])\cap I(1+h)\subset\Ad_{t^{-\mu}}L^-G\cap I(1+h)
$$

From the definition of $\phi$, we know it is generic on $I(1+h)/I(2+h)$. Thus it is nontrivial on $U_\alpha(kt)$. Since $a\in I(1)$, $\Ad_{a^{-1}}$ acts as identity on $I(1+h)/I(2+h)$, $\phi$ is also nontrivial on $\Ad_{a^{-1}}U_{\alpha_i}(kt)$. Thus such an orbit is irrelevant and we must have $\langle \alpha,\mu\rangle \geq 0$, $\forall\alpha\in\Delta$. Similarly, if $\langle \theta,\mu \rangle \geq 2$, then $\Ad_{a^{-1}}U_{-\theta}(kt^2)$ is contained in the stabilizer, and $\phi$ is nontrivial on it. Thus $\langle \theta,\mu\rangle \leq 1$, i.e. $\mu$ is minuscule.

\subsection{Part (ii): trivial coweight}\label{ss:rigidity for trivial coweight}
First consider the case of $\mu=1$. Proposition \ref{p:minuscule cosets} shows $L^-G I(1)=L^-GAJ$. We claim that if $L^-GaJ$, $a\in A$ is relevant, then $a=1$.

\subsubsection{}
Assume $a=\exp(et)$, $e\in\fa$. We need to show that if $L^-GaJ$ is relevant, then $e=0$. Note that
$$
\Ad_{a^{-1}}(U(k)\cap I(1+h/2))\subset\Ad_{a^{-1}}L^-G\cap J
$$
and $\phi$ must be trivial on above subgroup. Note that $U(k)\cap I(1+h/2)=\exp(\fu_{\geq 1+h/2})$. From the construction of $\phi$, we have
\begin{equation}\label{eq:unit coset relevant}
\begin{split}
&\phi(\Ad_{a^{-1}}(U(k)\cap I(1+h/2)))\\
=& \Res(\kappa(\tilde{X},\Ad_{a^{-1}}\fu_{\geq 1+h/2}))\frac{dt}{t}\\
=& \Res(\kappa(\Ad_{\exp(et)}(t^{-1}N^-+t^{-2}E_\theta),\fu_{\geq 1+h/2}))\frac{dt}{t}\\
=& \Res(\kappa(t^{-1}N^-+t^{-2}E_\theta+[e,N^-]+\frac{1}{2}[e,[e,E_\theta]]+t(\cdots),\fu_{\geq 1+h/2}))\frac{dt}{t}\\
=& \kappa([e,N^-]+\frac{1}{2}[e,[e,E_\theta]],\fu_{\geq 1+h/2})=0
\end{split}
\end{equation}

and we get
\begin{equation}\label{eq:simplified relevant}
\kappa([e,N^-]+\frac{1}{2}[e,[e,E_\theta]],\fu_{\geq 1+h/2})=0.
\end{equation}

\subsubsection{}
Assume $e=\sum_{r=1}^{h/2}e_r\in\fa$, $e_r\in\fa_r$. Since $\fa_1=0$, we already know $e_1=0$. We will prove $e_r=0$ by induction on $r$.

Suppose $2\leq r\leq h/2$ and $e_i=0$ for $1\leq i\leq r-1$. Note that $h+1-r\geq 1+h/2$ and $\fg(h+1-r)\subset\fu_{\geq 1+h/2}$. From \eqref{eq:simplified relevant}, we have
$$
\kappa([e,N^-]+\frac{1}{2}[e,[e,E_\theta]],\fg(h+1-r))=\kappa([e_r,N^-]+\frac{1}{2}\sum_{r_1+r_2=r}[e_{r_1},[e_{r_2},E_\theta]],\fg(h+1-r))=0
$$

Here we use that $e_i\in\fg(i-h)$ for any $i$ and that $\kappa(X,\fg(h+1-r))=0$ for any $X\not\in\fg(r-1-h)$. Since $1\leq r_1,r_2\leq h/2$ and $r_1+r_2=r$, we have $r_1,r_2<r$. By induction hypothesis, $e_{r_1}=e_{r_2}=0$. Thus
\begin{equation}\label{eq:1}
\kappa([e_r,N^-],\fg(h+1-r))=\kappa(e_r,[N^-,\fg(h+1-r)])=0
\end{equation}

Also, by the principal $\mathfrak{sl}_2$-triple decomposition, we have
\begin{equation}\label{eq:2}
\fg(h-r)=\ker(\ad_{N^+}|_{\fg(h-r)})\oplus[N^-,\fg(h+1-r)]
\end{equation}

Recall $e_r\in\fa_r=[\fg(r-1-h),N^+]$. Assume $e_r=[Y,N^+]$. We have
\begin{equation}\label{eq:3}
\kappa(e_r,\ker(\ad_{N^+}|_{\fg(h-r)}))=\kappa([Y,N^+],\ker(\ad_{N^+}|_{\fg(h-r)}))=\kappa(Y,[N^+,\ker(\ad_{N^+}|_{\fg(h-r)})])=0
\end{equation}

Combining \eqref{eq:1},\eqref{eq:2},\eqref{eq:3}, we get
$$
\kappa(e_r,\fg(h-r))=0.
$$

Also observe that $\kappa(e_r,\fg(-r))\subset\kappa(\fg(r-h),\fg(-r))=0$. Since $\fg_{-r}=\fg(h-r)\oplus\fg(-r)$, we get
$$
\kappa(e_r,\fg_{-r})=0
$$

As $e_r\in\fg_r$ and $\kappa$ is nondegenerate on $\fg_r\times\fg_{-r}$, we get $e_r=0$. By induction, $e=0$. This proves the claim.

\subsection{Part (iii): general minuscule coweights}
Next, we reduce the case of general $\mu$ to $\mu=1$.

We claim that if $L^-G t^\mu aJ$, $a\in A^\mu$, is relevant, then $a=1$.

\subsubsection{}
Similar as before, if a coset $L^-G t^\mu aJ$ is relevant, $a=\exp(et)\in A^\mu$, $e=\Ad_{t^{-\mu}w_P}e'$ where $e'\in\fa$, then $\phi$ must be trivial on
$$
\Ad_{a^{-1}}(\Ad_{t^{-\mu}}U_\mu\cap I(1+h/2))\subset\Ad_{a^{-1}t^{-\mu}}L^-G\cap J.
$$

We get
\begin{equation}\label{eq:minuscule relevant}
\begin{split}
&\phi(\Ad_{a^{-1}}(\Ad_{t^{-\mu}}U_\mu\cap I(1+h/2)))\\
=&\Res(\kappa(\Ad_a\tilde{X},\Ad_{t^{-\mu}}\fu_\mu\cap\Lie(I(1+h/2))))\frac{dt}{t}\\
=&\Res(\kappa(\Ad_{t^{-\mu}w_P}\Ad_{\exp(e't)}\Ad_{w_P^{-1}t^\mu}\tilde{X},\Ad_{t^{-\mu}w_P}\fu\cap\Lie(I(1+h/2))))\frac{dt}{t}\\
=&\Res(\kappa(\Ad_{\exp(e't)}\Ad_{w_P^{-1}t^\mu}\tilde{X},\fu\cap\Ad_{w_P^{-1}t^\mu}\Lie(I(1+h/2))))\frac{dt}{t}=0
\end{split}
\end{equation}

We show this is the same equation as \eqref{eq:unit coset relevant}.

\subsubsection*{Step 1} We have to prove that
\begin{equation}\label{eq:reduction1}
\Ad_{w_P^{-1}t^\mu}\tilde{X}=\tilde{X}.
\end{equation} 

Equivalently, we need to show $\Ad_{t^{-\mu}w_P}\tilde{X}=\tilde{X}$. In fact, using Lemma \ref{l:w_P0w_0}.(1) and the proof of Lemma \ref{l:w_P0w_0}.(2), we have

\begin{align*}
&\Ad_{t^{-\mu}w_P}\tilde{X}
=\Ad_{t^{-\mu}w_{P,0}}\Ad_{w_0}(t^{-1}N^-+t^{-2}E_\theta)
=\Ad_{t^{-\mu}w_{P,0}}(t^{-1}\sum_{\alpha\in\Delta}E_\alpha+t^{-2}E_{-\theta})\\
=&\Ad_{t^{-\mu}}\Ad_{w_{P,0}}(t^{-1}\sum_{\alpha\in\Delta-\{\alpha_\mu\}}E_\alpha+t^{-1}E_{\alpha_\mu}+t^{-2}E_{-\theta})\\
=&\Ad_{t^{-\mu}}(t^{-1}\sum_{\alpha\in\Delta-\{\alpha_\mu\}}E_{-\alpha}+t^{-1}E_\theta+t^{-2}E_{-\alpha_\mu})\\
=&t^{-1}\sum_{\alpha\in\Delta-\{\alpha_\mu\}}E_{-\alpha}+t^{-2}E_\theta+t^{-1}E_{-\alpha_\mu}=\tilde{X}.
\end{align*}
Here the action of Weyl group sending $E_\alpha$'s to other $E_\alpha$'s comes from the choice of $w_P$ in Lemma \ref{l:w_P0w_0}.(2).

\subsubsection*{Step 2} We find that
\begin{equation}\label{eq:reduction2}
\begin{split}
&\fu\cap\Ad_{w_P^{-1}t^\mu}\Lie(I(1+h/2))
=\Ad_{w_P^{-1}t^\mu}(\Ad_{t^{-\mu}w_p}\fu\cap\Lie(I(1+h/2)))\\
=&\Ad_{w_P^{-1}t^\mu}((\fu_M\oplus\fu_P^-t)\cap\Lie(I(1+h/2)))\\
=&\Ad_{w_P^{-1}t^\mu}((\fu_M\cap\bigoplus_{r=1+h/2}^{h-1}\fg_r)\oplus(\fu_P^-\cap\bigoplus_{r=1+h/2}^{h-1}\fg_r)t)\\
=&\Ad_{w_P^{-1}}((\fu_M\cap\bigoplus_{r=1+h/2}^{h-1}\fg_r)\oplus(\fu_P^-\cap\bigoplus_{r=1+h/2}^{h-1}\fg_r))\\
=&\Ad_{w_P^{-1}}(\fu_\mu\cap\bigoplus_{r=1+h/2}^{h-1}\fg_r)=\fu_{\geq1+h/2}.
\end{split}
\end{equation}

In the last step above we used Lemma \ref{l:u conju u_mu} and Lemma \ref{l:w_P0w_0}.(3).

From \eqref{eq:reduction1} and \eqref{eq:reduction2}, we see that \eqref{eq:minuscule relevant} is same as \eqref{eq:unit coset relevant}. Thus from the proof in \S\ref{ss:rigidity for trivial coweight}, we get $e'=0$, $a=1$. The Claim is proved.

\subsection{Part (iv): the minuscule orbits are relevant} \label{ss:relevant}
To complete the proof of Theorem \ref{t:rigidity}, it remains to examine that $L^-G t^\mu J$ is relevant when $\mu$ is minuscule. From Lemma \ref{l:stab=u_mu}, this is same as requiring that the character $\chi$ in Definition \ref{d:Airy} is trivial on
$$
\Ad_{t^{-\mu}}L^-G\cap J=(\Ad_{t^{-\mu}}L^-G\cap I(1))\cap J=(\Ad_{t^{-\mu}}U_\mu)\cap S(1)I(1+h/2).
$$

It suffices to show $(\Ad_{t^{-\mu}}U_\mu)\cap S(1)I(1+h/2)=(\Ad_{t^{-\mu}}U_\mu)\cap I(1+h/2)$, because $\phi=\chi|_{I(1+h/2)}$ is trivial on this subgroup from its definition. To see this, we can prove the following fact:
$$
(\Ad_{t^{-\mu}}U_\mu)\cap S(r)I(1+h/2)=(\Ad_{t^{-\mu}}U_\mu)\cap S(r+1)I(1+h/2),\quad \forall\ 1\leq r\leq h/2.
$$

Given $sg\in (\Ad_{t^{-\mu}}U_\mu)\cap I(r)$ on the left-hand side, $s\in S(r)$, $g\in I(1+h/2)$, we look at its image $\bar{s}$ in the quotient $I(r)/I(r+1)\simeq\fg_r$. By Lemma \ref{l:g_r decomp}, we get $\bar{s}=\fz_r\cap\fu_{\mu,r}=0$. Thus $sg\in I(r+1)$, $s\in S(r+1)$. Applying this fact for $r=1,2,...,h/2$, we get the desired $(\Ad_{t^{-\mu}}U_\mu)\cap S(1)I(1+h/2)=(\Ad_{t^{-\mu}}U_\mu)\cap I(1+h/2)$. 

This finishes the proof of rigidity of the Airy automorphic datum.

\section{Geometry of relevant orbits}\label{s:geomrelevantorb}
In this section we study the geometry of the relevant orbits. They are not open, but we prove they are closed in the generic locus of their respective connected components. In particular we show that relevant orbits are locally closed in the whole stack $\Bun_{J}$ and that it is affine in the neutral component. The main result is Proposition \ref{p:geom of relevant orbits}.

For a subgroup $H\subset G(F)$ of finite codimension in some parahoric subgroup of $G(F)$ we write $\Bun_H:=[L^-G\backslash LG/H]$. Recall the Birkhoff decomposition \eqref{eq:orbits}. 

\begin{prop}\ \label{p:geom of relevant orbits}
\begin{itemize}
	\item [(1)] The embedding $L^-G\backslash L^-G t^\mu I(1)/J\subset\Bun_J$ is open for any minuscule coweight $\mu$. For $\mu=1$ it's defined by the non-vanishing of a section of line bundle and in particular it is affine. 
	\item [(2)] The embedding of the relevant orbit $L^-G\backslash L^-G t^\mu J/J\subset L^-G\backslash L^-G t^\mu I(1)/J$ is closed for any minuscule coweight $\mu$. 
\end{itemize}
\end{prop} 
The rest of this section concerns the proof of this proposition. For (1), note that since $J\subset I$, we have map $\Bun_J\ra\Bun_I=[L^-G\backslash LG/I]$. Thus the openness of 
\[
L^-G\backslash L^-G t^\mu I(1)/J=L^-G\backslash L^-G t^\mu I/J\subset\Bun_J
\]
 for minuscule $\mu$ follows from Corollary \ref{c:open orbits on Fl_G}. Moreover we know from \cite[Lemma 3.1]{YunMotive} or \cite[Corollary 1.3]{HNY13} that the embedding of the trivial bundle 
\[
[L^-G\backslash L^-G\cdot G(\cO)/G(\cO)]\subset\Bun_{G(\cO)}
\] is defined by the non-vanishing of a section of a line bundle. Thus the same holds for the embedding of its preimage into $\Bun_J$. Thus, 
\[
L^-G\backslash L^-G\cdot G(\cO)/J=L^-G\backslash L^-G\ G(k)G(tk[\![t]\!])I(1)/J=L^-G\backslash L^-G I(1)/J\subset\Bun_J,
\]
is also affine and open.

To prove (2), we need some notation. For simplicity, we assume that $h$ is even. The odd $h$ case requires only minor modifications. For a subgroup $K\subset I(1)$, we denote its image in $I(1)/I(1+h/2)$ by $\overline{K}$. Denote
$$
U_{\leq h/2}(k):=\prod_{1\leq\Ht\alpha\leq h/2} U_\alpha(k),\quad U^-_{\leq-h/2}(kt):=\prod_{\Ht\alpha\leq-h/2}U_\alpha(kt).
$$

Regarding $U_{\leq h/2}(k)$ and $U^-_{\leq-h/2}(kt)$ as subgroups in $I(1)/I(1+h/2)$, we have 
$$
I(1)/I(1+h/2)\simeq U_{\leq h/2}(k)\ltimes U^-_{\leq-h/2}(kt)
$$ 
where $U^-_{\leq-h/2}(kt)$ is a commutative subgroup. Thus $\fu^-_{\leq-h/2}\simeq\Lie(U^-_{\leq-h/2})(kt)\simeq U^-_{\leq-h/2}(kt)$ in notations of \S\ref{sss:cyclic grading}. Denote the projection from $I(1)/I(1+h/2)$ to $U_{\leq h/2}(k)$ (resp. $U^-_{\leq-h/2}(kt)$) by $p_+$(resp. $p_-$). From Proposition\ref{p:minuscule cosets}, we can see that $I(1)/I(1+h/2)=U_{\leq h/2}(k)A(S(1)/S(1+h/2))=U_{\leq h/2}A\overline{S(1)}$.

Returning to the proposition,  let $\mu=1$. Since $L^-G\cap I(1)=U(k)$, we have $L^-G\backslash L^-G I(1)/J=U(k)\backslash I(1)/J$. From Lemma \ref{l:lower part of S(1)} (below), we see that $U_{\leq h/2}(k)S(1)$ is closed in $I(1)/I(1+h/2)$ and $U_{\leq h/2}(k)$ is closed in $I(1)/(S(1)I(1+h/2))=I(1)/J$. Since $U_{\leq h/2}(k)$ is the preimage of the unit coset $\overline{1}\in U(k)\backslash I(1)/J$ under the quotient map from $I(1)/J$, we see $\overline{1}\in U(k)\backslash I(1)/J$ is closed. This proves the statement for $\mu=1$. For other minuscule coweights, it suffices to deduce from the $\mu=1$ case by applying the conjugation by $t^{-\mu}w_P$ where $w_P$ is defined in Lemma \ref{l:w_P0w_0}.  To finish the proof of the proposition, it remains to prove the following: 

\begin{lem}\label{l:lower part of S(1)}
	There exists a subset of roots $\Phi_S\subset\Phi(U^-_{\leq-h/2})=:\Phi_{\leq-h/2}$ and polynomials $f_\beta\in k[x_\alpha\mid\alpha\in\Phi_S]$ for each $\beta\in\Phi_{\leq-h/2}-\Phi_S$ such that
	$$
	p_-(\overline{S(1)})=\{\sum_{\alpha\in\Phi_S}\lambda_\alpha E_\alpha+\sum_{\beta\in\Phi_{\leq-h/2}-\Phi_S}f_\beta(\lambda_\alpha)E_\beta\mid\lambda_\alpha\in k,\forall \alpha\in\Phi_S \}\subset\fu^-_{\leq-h/2}\simeq U^-_{\leq-h/2}(kt)
	$$
	In particular, $p_-(\overline{S(1)})\simeq\bA^{|\Phi_S|}$ is closed in $U^-_{\leq-h/2}(kt)$.
\end{lem}
\begin{proof}
Recall that $S(1)=I(1)\cap\Stab_{G(F)}(t^{-1}N^-+t^{-2}E_\theta)$ and $\fz_r=\Stab_{\fg_r}(N^-+E_\theta)$. We take the obvious choice of the section of $\fz_r\subset\fg_r\simeq\Lie(I(r)/I(r+1))$ in $\Lie(I(r))$, which we still denote by $\fz_r$. Then $\exp(\fz_r)\subset S(1)$ for any $r>0$. 

Claim: $\exp(\oplus_{r=1}^{h/2}\fz_r)=S(1)/S(1+h/2)$ as subgroups in $I(1)/I(1+h/2)$. First, the inclusion $\exp(\oplus_{r=1}^{h/2}\fz_r)\subset S(1)/S(1+h/2)$ is clear from definition. Conversely, given any $s\in S(1)$, consider its image $\overline{s}$ in $S(1)/S(2)\simeq\fz_1\subset \fg_1\simeq I(1)/I(2)$. Thus $\overline{s}=\exp(v_1)\in I(1)/I(2)$ for some $v_1\in\exp(\fz_1)$ and $\exp(-v_1)s\in S(2)$. Similarly, there is $v_2\in\fz_2$ such that $\exp(-v_1-v_2)s\in S(3)$. Inductively, this proves the claim.

Now given any $\overline{s}=\exp(\sum_{r=1}^{h/2}v_r)\in S(1)/S(1+h/2)$ where $v_r\in\fz_r$. Recall $\fz_r\subset\fg_r=\fg(r)\oplus\fg(r-h)$. Denote the projections of $v_r$ to $\fg(r)$ and $\fg(r-h)$ by $v_r^+,v_r^-$ respectively. Denote $v=\sum_{r=1}^{h/2}v_r$, $v^+=\sum_{r=1}^{h/2}v_r^+$, $v^-=\sum_{r=1}^{h/2}v_r^-$. We will see below that $\exp(-v^+)\overline{s}=p_-(\overline{s})$. In fact, by Baker-Campbell-Hausdorff formula \footnote{This is applicable because we are working in the finite quotient $I(1)/I(1+h/2)$ in which the exponential map is algebraic.}, we have
\begin{equation}\label{eq:p_-(s)}
\begin{split}
\exp(-v^+)\overline{s}&=\exp(-v^+)\exp(v^++v^-)\\
&=\exp(v^-+\frac{1}{2}[-v^+,v^++v^-]+\frac{1}{12}[-v^+,[-v^+,v^++v^-]]+\cdots)\\
&=\exp(v^-+\sum_{j=1}^{h/2}c_j\ad_{v^+}^j(v-))=p_-(\overline{s})
\end{split}
\end{equation}

In the above last step, $\ad_{v^+}^j(v-)\in\Lie(I(1+h/2))$ when $j>h/2$\footnote{The denominator of one iterated Lie bracket in the Campbell-Hausdorff formula is a product of numbers smaller or equal to the number of Lie brackets in this term plus one, thus here the denominator has prime factors at most $h/2+1$. This is covered by our assumption on the characteristic.}. Thus we want to show that when $v$ vary, the quotient of $v^-+\sum_{j=1}^{h/2}c_j\ad_{v^+}^j(v-)$ in $\fu^-_{\leq-h/2}$ is as in the statement of lemma. 

Now we construct the subset of roots $\Phi_S$ in the statement. For each $1\leq r\leq h/2$, recall from Lemma \ref{l:p_-(z_r)} that $v^-\in\ker(\ad_{N^-}|_{\fg(r-h)})=p_-(\fz_r)$. We have seen in the proof of Lemma \ref{l:p_-(z_r)} that $\dim\ker(\ad_{N^-}|_{\fg(r-h)})=\dim\fz_r=\#\{e_j=r\}$, i.e. the number of exponents that equal $r$. Thus $\dim\ker(\ad_{N^-}|_{\fg(r-h)})=1$ whenever it is not zero, except in type $D$ with even rank where the dimension is two. We choose a basis $0\neq Y_r\in\ker(\ad_{N^-}|_{\fg(r-h)})$ for each $r$ that equals some exponent, and in the type $D$ case with even rank case we choose a basis $Y_{h/2}^1,Y_{h/2}^2\in\ker(\ad_{N^-}|_{\fg(-h/2)})$. For each $Y_r=\sum_{\alpha\in\Phi(\fg(r-h))}\mu_\alpha E_\alpha$, we choose a root $\beta_r\in\Phi(\fg(r-h))$ such that $\mu_{\beta_r}\neq0$. For the type $D$ exception, write $Y_{h/2}^i=\sum_{\alpha\in\Phi(\fg(-h/2))}\mu^i_\alpha E_\alpha$ for $i=1,2$. Since $Y_{h/2}^1, Y_{h/2}^2$ are linearly independent, we can find two roots $\beta_{h/2}^1,\beta_{h/2}^2$ such that the components of $Y_{h/2}^1, Y_{h/2}^2$ in these two root subspaces are also linearly independent (because the matrix of $\mu_\alpha^1,\mu_\alpha^2$'s has rank two). Let $\Phi_S$ be the union of these $\beta_r$ and $\beta_{h/2}^i$.

Let $v^-=\sum_{1\leq r=e_j\leq h/2}\lambda_r Y_r$. From Lemma \ref{l:p_-(z_r)}, we know that the projection from $v=v^++v^-$ to $v^-$ is an isomorphism. Thus $v^+$ is determined by $v^-$ and we can write $v^+$ as a linear combination of root vectors $E_\alpha$ whose coefficients are polynomials in $\lambda_r$'s. More precisely, $v^+$ is a linear combination of root vectors $E_\alpha$ where the height of $\alpha$ is some exponent $1\leq\Ht(\alpha)=e_j\leq h/2$, and the coefficient of $E_\alpha$ in $v^+$ is a polynomial in $\lambda_{e_j}$. Thus we can write the last term in \eqref{eq:p_-(s)} as
\begin{equation}\label{eq:p_-(s) 2}
p_-(\overline{s})=v^-+\sum_{j=1}^{h/2}c_j\ad_{v^+}^j(v^-)=\sum_{\alpha\in\Phi(\fu^-_{\leq-h/2})}f_\alpha E_\alpha
\end{equation}

where $f_\alpha$ is a polynomial in $\lambda_{e_j}$ for exponents $1\leq e_j\leq \Ht\alpha$. Let $\lambda_\alpha:=f_\alpha$ for $\alpha\in\Phi_S$. It remains to show that 
\begin{itemize}
	\item $f_\alpha$ for $\alpha\in\Phi_{\leq-h/2}-\Phi_S$ can be expressed as polynomials in $\lambda_\alpha$ for $\alpha\in\Phi_S$,
	\item the change of variables from $\lambda_{e_j}$'s to $\lambda_\alpha$ for $\alpha\in\Phi_S$ is invertible and algebraic.
\end{itemize}

We show this by induction, i.e. we show that the above is true when we restrict to $\Phi_{\leq R-h}$ for any $1\leq R\leq h/2$. This means for any $\beta\in\Phi_{\leq R-h}$, $f_\beta$ is a polynomial of $\lambda_\alpha$ for $\alpha\in\Phi_S\cap\Phi_{\leq R-h}$, and the change of variables from $\lambda_{e_j}$ for $1\leq e_j\leq R$ and $\lambda_\alpha$ for $\alpha\in\Phi_S\cap\Phi_{\leq R-h}$ has an inverse defined by polynomials.

First when $R=1$ we have $\lambda_1=f_{-\theta}=\lambda_{-\theta}$ and the statement is true. Next assume the statement is true on $\Phi_{\leq R-h}$, $1\leq R<h/2$. 
Note from the formula \eqref{eq:p_-(s) 2}
that $f_\alpha$ is always a polynomial in $\lambda_{e_j}$ for $1\leq e_j-h<\Ht(\alpha)$, together with a linear term in $\lambda_{\Ht(\alpha)+h}$ when $\Ht(\alpha)+h$ is an exponent. Thus for any $\alpha$ with $\Ht(\alpha)=R+1-h$, if $R+1$ is not an exponent, then the statement is clear on $\Phi_{\leq R+1-h}$ by the induction hypothesis. If $R+1=e_j$ is an exponent, first assume we are not in the exceptional case of type $D$ with even rank and $R+1=h/2$. Recall $\Phi_S\cap\Phi(\fg(R+1-h))=\{\beta_{R+1}\}$. From the definition of $\beta_{R+1}$ and by the induction hypothesis, we have
$$
\lambda_{\beta_{R+1}}=f_{\beta_{R+1}}=c\lambda_{e_j}+g,\quad c\neq 0,\ g\in k[\lambda_\alpha\mid\alpha\in\Phi_S\cap\Phi_{\leq R-h}].
$$

Thus the change of variables for $R+1$ is invertible with algebraic inverse. Also from the previous discussion we see that any $f_\alpha$ for $\beta_{R+1}\neq\alpha\in\Phi(\fg(R+1-h))$ is a polynomial in $\lambda_r$ for $r\leq R+1$, thus also a polynomial in $\lambda_{\beta_r}$ for $r\leq R+1$. Thus the statement is proved. 

In the exceptional case of type $D$ with even rank, $R+1=h/2$, $\lambda_{h/2}$ is replaced by two coefficients $\lambda_{h/2}^1,\lambda_{h/2}^2$ of $Y_{h/2}^1,Y_{h/2}^2$. We have system of equations
$$
\begin{cases}
\lambda_{\beta_{h/2}^1}=f_{\beta_{h/2}^1}=\mu^1_{\beta_{h/2}^1}\lambda_{h/2}^1+\mu^1_{\beta_{h/2}^2}\lambda_{h/2}^2+g_1, \quad g_1\in k[\lambda_\alpha\mid\alpha\in\Phi_S\cap\Phi_{\leq-1-h/2}];\\
\lambda_{\beta_{h/2}^2}=f_{\beta_{h/2}^2}=\mu^2_{\beta_{h/2}^1}\lambda_{h/2}^1+\mu^2_{\beta_{h/2}^2}\lambda_{h/2}^2+g_2,  \quad g_2\in k[\lambda_\alpha\mid\alpha\in\Phi_S\cap\Phi_{\leq-1-h/2}].\\
\end{cases}
$$

Recall from the definition of $\beta_{h/2}^i$ that the matrix of coefficients $\mu_{\beta_{h/2}^j}^i$ is invertible. Thus the statement is true in this case. This completes the proof of the lemma.
\end{proof}

\section{Geometric Hecke operators and eigensheaves}\label{s:heckeeigen}

Unless otherwise stated we will base change to an algebraic closure $\overline{k}$ of $k$ in this section without changing notation. 

\subsection{Automorphic sheaves} The character $\chi : J \rightarrow \bQlt$ determines a (rank one) character sheaf $K_{\chi}$ on $J$ that descends to the quotient $L=J/J^{+}$, cf. \cite{Masoud} or \cite[Appendix A]{Yun14}. Let $S$ be any scheme over $k$. We denote by $\cD(S,\chi)$ the derived category of $(J,K_\chi)$-equivariant $\bQl$-complexes on $S \times \Bunpl$ (where the action on $S$ is trivial) and by $\Perv(S,\chi)$ the full abelian subcategory of perverse sheaves. We'll write  $\cD(\chi)$ and $\Perv(\chi)$ for complexes and perverse sheaves on $\Bunpl$. For $\alpha \in \Omega$ we denote complexes and perverses sheaves on $S\times \Bunpl^{\alpha}$ by $\cD(S,\chi)_{\alpha}$ and $\Perv(S,\chi)_{\alpha}$ respectively. We'll denote the relevant orbit in $\Bunpl^{\alpha}$ by $O_{+}^{\alpha}$.


\begin{lem} \label{l:cleanext} \begin{enumerate}
\item For all $\alpha \in \Omega$ the category $\Perv(\chi)_{\alpha}$ contains a unique simple object $\cA_{\alpha}$
which is a clean extension from the relevant orbit $O^{\alpha}_{+}$. 
\item Let $S$ be a scheme of finite type over $k$. The assignment 
\begin{align*}  D^b(S) &\to \cD(S,\chi)_{\alpha} \\ 
K &\mapsto K \boxtimes \cA_{\alpha}
\end{align*}
is an equivalence of categories that is $t$-exact for the perverse $t$-structure.
\end{enumerate}
\end{lem}
\begin{proof} (1) Using the identification of connected components from \S\ref{sss:comp} it suffices to consider the case $\alpha=0$. We will mostly drop the subscript indicating the component for sake of notation. By Proposition \ref{p:geom of relevant orbits} the embedding $j: O_{+} \hookrightarrow \Bunpl^{0}$ of the relevant orbit is locally closed and factors into a closed embedding $i : O_{+} \hookrightarrow \Bunpl^{\circ}$ followed by the open embedding $ \Bunpl^{\circ}\hookrightarrow \Bunpl^{0}$. Fix a bundle $\cE_{+} \in O_{+}$ and let $\cU$ be its stabilizer under the $J$-action. This is a connected unipotent group by \ref{ss:relevant}. We have an isomorphism of quotient stacks 
\[ \left[O_+ / J \right] \cong \left[\mathrm{pt}/ \cU \right] \]
and by Theorem \ref{t:rigidity} any object $K\in \cD(\chi)$ has vanishing stalks and costalks outside the relevant orbit $O_{+}$. Hence restriction along $j$ and to the point $\cE_{+}$ defines a $t$-exact equivalence
\[ 
\cD(\chi)_0 \cong D^{b}_{(J,K_{\chi})}(O_{+}) \cong D^{b}_{(\cU,K_{\chi}|_{\cU})}(\mathrm{pt}).
\]
The latter category is furthermore equivalent to $D^b(\Rep(\pi_0(\cU))) \cong \gr\mathrm{Vect}$, see \cite[Lemma A.4.4.]{Yun14}. In particular $\Perv(\chi)_{0} \simeq \mathrm{Vect}$ contains a unique simple object $\cA_{0}$.

(2) Let $i_{\cE_{+}}$ be the embedding of $\cE_{+}$ into $\Bunpl$. By (1) pullback along $\id_S \times i_{\cE_{+}}$ is a quasi-inverse. \end{proof}

\begin{defe}  We let $\cA_\chi$ be the perverse sheaf on $\Bunpl$ that is given on the component $\Bunpl^{\alpha}$ by $\cA_{\alpha}=\bfT_{\alpha,!}\cA_{0}$ where $\bfT_{\alpha}: \Bunpl^{0} \cong \Bunpl^{\alpha}$ is the isomorphism described in (\ref{eq:infHecke}).
\end{defe}

\subsection{Geometric Hecke operators}
We recall the definition and some properties of geometric Hecke operators. The Hecke stack $\Hk$ classifies tuples
$(\cE, \cE', x, \phi)$ where $\cE, \cE' \in \Bunpl$, $x\in \bA^1$ and $\phi : \cE|_{X\setminus \{x\}} \xrightarrow{\cong} \cE|_{X\setminus \{x\}}$ is an isomorphism respecting the level structure. The Hecke correspondence is given by

\[\begin{tikzcd}
& \Hk  \ar{dl}[swap]{\hl} \ar{dr}{\pi \times \hr} & \\
\Bunpl & & \bA^1 \times \Bunpl 
\end{tikzcd}\]

where $\hl$ and $\hr$ map $(\cE, \cE', x, \phi)$ to $\cE$ and $\cE'$ respectively and $\pi$ maps it to $x$.
\begin{rem} The morphism $\pi \times \hr$ is a locally trivial fibration with fibers non-canonically isomorphic to the affine Grassmannian. In particular it is ind-proper. Similarly $\hl$ is a locally trivial fibration with fibers isomorphic to the Beilinson-Drinfeld Grassmannian. 
\end{rem}

The geometric Satake equivalence associates to each $V \in \Rep(\hG)$ an $L^+G \rtimes \Aut(k[[t]])$-equivariant intersection cohomology sheaf $\IC_V$ on the affine Grassmannian $\Gr_G$. This in turn determines a sheaf $\IC_V^{\Hk}$ on $\Hk$ which is isomorphic to the intersection cohomology sheaf along the fibers of $\pi \times \hr$. 
\begin{defe} For $V\in \Rep(\hG)$ the geometric Hecke operator $\bfT_V$ is the functor 
\[ \bfT_V = (\pi \times \hr)_! (\hr^*( - ) \otimes \IC_V ) : \cD(\chi) \rightarrow \cD(\bA^1, \chi). \]
\end{defe}

More generally for any finite set $I$ there is an iterated Hecke stack $\Hk_I$ classifying simultaneous modifications of bundles at sets of points indexed by $I$. Using $\Hk_I$ one defines for any $V_I\in \Rep(\hG^I)$ and any $k$-scheme $S$ the iterated Hecke operators 
\[\bfT_{V_{I},S} : \cD(S,\chi) \rightarrow D^b_{(J,K_\chi)}((\bA^1)^{I} \times S \times \Bunpl). \]
They are functorial in $V_{I}$ and carry factorization structures: 
\begin{enumerate}
\item Denoting by $\mathrm{triv}\in \Rep(\hG^I)$ the trivial representation, there is a canonical isomorphism $\bfT_{\mathrm{triv},S}(K) \cong K\boxtimes \bQl$.
\item For two finite sets $I,J$ and representations $V_I\in \Rep(\hG)$, $V_J\in \Rep(\hG)$ we have a canonical isomorphism
\[\bfT_{{V_I}\boxtimes {V_J},S} \cong \bfT_{V_{I}, (\bA^1)^{J} \times S} \circ \bfT_{V_{J},S}. \]
\item For any surjection $J \twoheadrightarrow I$ we have diagonal maps $(\bA^1)^{I} \hookrightarrow (\bA^1)^{J}$ and $\hG^{I} \hookrightarrow \hG^{J}$ and there is a canonical isomorphism
\[\bfT_{V_J,S} |_{(\bA^1)^{I} \times S \times \Bunpl} \cong \bfT_{V_{J}|_{\hG^{I} },S}. \]
\end{enumerate}
The above isomorphisms are subject to compatibility conditions for which we refer to \cite{Ga07}. 

\begin{defe} A Hecke eigensheaf with eigenvalue $E$ for the Airy automorphic datum is a triple $(A,E,\varphi)$ where $A \in \Perv(\chi)$, 
$E$ is a $\hG$-local system on $\bA^1$ and $\varphi$ is a collection of isomorphisms 
$$
\varphi_V: \bfT_V(A) \cong E_V \boxtimes A
$$
 compatible with the symmetric tensor structure on $\Rep(\hG)$ and the factorization structure of iterated Hecke operators indicated above. For details we again refer to \cite{Ga07}.
\end{defe}

\begin{thm} \label{t:heckeeigen}
\begin{enumerate}
\item For any $V\in \Rep(\hG)$ the functor $\bfT_V(-)[1]: \cD(\chi) \rightarrow \cD(\bA^1,\chi)$ is exact for the perverse $t$-structure. 
\item The perverse sheaf $\cA_\chi$ can be extended to a Hecke eigensheaf with eigenvalue $\mathrm{Ai}_{\hG}$ that descends to the finite field $k$ and such that $\mathrm{Ai}_{\hG,V}$ is a semisimple local system for each $V$ in $\Rep(\hG)$
\end{enumerate}
\end{thm}
\begin{proof} 
(1) The proof is similar to the proof of \cite[Theorem 3.8.]{Yun16} and \cite[Remark 4.2.]{HNY13}. Since the connected components of $\Bunpl$ are all isomorphic, we will restrict to the case of a perverse sheaf supported on the neutral component $\Bunpl^{0}$. By Lemma \ref{l:cleanext} any such perverse sheaf $A$ is a clean extension from the locally closed relevant orbit. In particular it is of the form $A=j_*B\cong j_!B$ for a perverse sheaf $B$ on the generic open locus $j:\Bunpl^{\circ}\hookrightarrow \Bunpl^{0}$. The key-point of the proof is that for the pre-image $j': \Hk^{\circ} \hookrightarrow \Hk$ of $\Bunpl^{\circ}$ under $\hl$ the morphism $(\pi \times \hr) \circ j'$ is ind-affine. This follows from Proposition \ref{p:geom of relevant orbits} and for further details we refer to the aforementioned references.

(2) Let $\nu$ be the central character of $V$. By Lemma \ref{l:cleanext} for any $\alpha\in \Omega$ we have $\bfT_V(\cA_{\alpha})[1] = \cF_{\alpha}^{V}[1] \boxtimes \cA_{\alpha+\nu}$ for some perverse sheaf $\cF_{\alpha}^{V}[1]$ on $\bA^1$. The argument from \cite[\S 4.2.]{HNY13} using the isomorphisms in (\ref{eq:infHecke}) shows that $\cF_{\alpha}^{V}[1]$ is independent of $\alpha$ and we will drop this index. Proving that each $\cF^{V}$ is a semisimple local system now works the same as in the proof of \cite[Lemma 4.4.6.]{Yun14}. For descent to the finite field $k$ see the next remark. 
\end{proof}

\begin{rem}
The situation fits into the general framework of \cite[Appendix A.2.]{JY20}. The category $\Perv(\chi)$ is a factorizable $\Rep(\hG)$-module category with coefficients in the category of local systems on $\bA^1$ in the sense of \cite[Appendix A.2.]{JY20}. Since the Hecke operators mix up connected components transitively the category $\Perv(\chi)$ is necessarily indecomposable as $\Rep(\hG)$-module category. The existence of the Hecke eigenvalue is therefore guaranteed by \cite[Theorem A.4.1.]{JY20}, see also \cite[Remark 4.2.4.(4)]{JY20}. Note that even though the category $\Perv(\chi)$ may contain infinitely many distinct simple objects (one on each connected component of $\Bunpl$), we have seen that $\bfT_V(\cA_{\alpha}) = \cF^{V} \boxtimes \cA_{\alpha+\nu}$ (here $\nu$ denotes the central character of $V$) for a local system $\cF$ independent of $\alpha$. In this situation  \cite[Theorem A.4.1.]{JY20} is still applicable. In particular since the embedding of the relevant orbit $O_{\alpha}^{+} \hookrightarrow \Bunpl^{\alpha}$ is defined over $k$, so is the automorphic sheaf $\cA_\chi$. We can then use the induced Frobenius-equivariant structure on the $\Rep(\hG)$-module category $\Perv(\chi)$ to obtain a Weil structure on $\mathrm{Ai}_{\hG}$, see \cite[Remark A.4.3.]{JY20}.
\end{rem}

\subsection{Rigidity of Airy sheaves} \label{ss:Conjectures}
Let us first recall the definition of rigidity of $\hG$-local systems, compare to \cite[\S 3.2.4]{Yun14}.  Let $j: U\hookrightarrow X=\bP^1_{k}$ be an open subset and fix a geometric point $u\in U$. Let $\cL$ be a ${\hG}$-local system on $U$ and let $\pi_{1}(U,u)$ be the étale fundamental group of $U$. 

Let $\widehat{\fg}^{\mathrm{der}}$ be the derived Lie algebra of $\widehat{\fg}$ and consider the representation
\[\Ad : \hG(\bQl) \rightarrow \GL(\widehat{\fg}^{\mathrm{der}}). \]
The local system $\cL$ corresponds to a homomorphism $\rho: \pi_{1}(U,u) \to \hG(\bQl)$ and we denote by $\cL^{\Ad}$ the local system corresponding to the composition $\Ad\circ \rho$.
\begin{defe}
We say that $\cL$ is \textup{cohomologically rigid} if 
\[\textup{H}^1 (X, j_{!*}\cL^{\Ad}) = 0 \]
where $j_{!*}$ denotes the minimal extension. 
\end{defe}

\begin{conj} The Airy local system $\mathrm{Ai}_{\hG}$ is cohomologically rigid. More specifically we predict the following. Let $n$ be the semisimple rank of $G$.
\begin{enumerate} 
\item For the Swan conductor of $\mathrm{Ai}_{\hG}^{\Ad}$ at $\infty$ we have $\mathrm{Sw}_{\infty}(\widehat{\fg}^{\mathrm{der}}) =  n(h+1)$.
\item Let $\cI_{\infty}$ be the inertia group at $\infty$ in the Weil group $W_{F}$ of $F$. Then $(\widehat{\fg}^{\mathrm{der}})^{\rho(\cI_\infty)} = 0$. 
\end{enumerate}
\end{conj}
Prediction (1) and (2) together imply that $\mathrm{Ai}_{\hG}$ is cohomologically rigid. Indeed by \cite[Lemma 3.2.6]{Yun14} we have an exact sequence
\[0 \rightarrow \textup{H}^{0}(\bA^1,\mathrm{Ai}_{\hG}^{\Ad}) \to (\widehat{\fg}^{\mathrm{der}})^{\rho(\cI_\infty)} \to \textup{H}_{c}^{1}(\bA^1,\mathrm{Ai}_{\hG}^{\Ad}) \to \textup{H}^{1}(X,j_{!*}\mathrm{Ai}_{\hG}^{\Ad} )\to 0
\]
where $j: \bA^1 \hookrightarrow X$ denotes the open embedding.  Now by (2) we have $(\widehat{\fg}^{\mathrm{der}})^{\rho(\cI_\infty)}=0$ and $ \textup{H}^{0}(\bA^1,\mathrm{Ai}_{\hG}^{\Ad})=0$, so to prove that $\mathrm{Ai}_{\hG}$ is cohomologically rigid it suffices to prove that $\textup{H}_{c}^{1}(\bA^1,\mathrm{Ai}_{\hG}^{\Ad})=0$. By the Grothendieck-Ogg-Shafarevich formula we have 
\[\textup{H}_{c}^{1}(\bA^1,\mathrm{Ai}_{\hG}^{\Ad}) = \dim(\widehat{\fg}^{\mathrm{der}}) - \mathrm{Sw}_{\infty}(\widehat{\fg}^{\mathrm{der}}).\]
Using (1) we get $\mathrm{Sw}_{\infty}(\widehat{\fg}^{\mathrm{der}}) = n(h+1) = n+ \#\Phi(\widehat{\fg}) = \dim(\widehat{\fg}^{\mathrm{der}})$, proving the claim.

\subsubsection{}Some evidence for this conjecture is provided in the following section in which we compute the eigen local system for $G=\GL_n$. It turns out that it coincides with a classical Airy sheaf defined by Katz in \cite{Katz87}. These sheaves are not only cohomologically rigid, but also physically rigid in the sense that they are determined by their local monodromy at $\infty$ up to isomorphism. The above conjecture is easy to check for these sheaves.

\section{The eigenvalue for $\GL_n$}\label{s:Identification}
Let $G=\GL_n$ and let $\cA_\chi$ be the Hecke eigensheaf obtained in \ref{t:heckeeigen} with eigenvalue $\Ai_{\hG}$. Let $V=\St$ be the standard representation of $G$. In this section we compute the eigenvalue $\cE_{\chi}=\Ai_{\GL_n}^{\St}$ following the recipe in \cite[\S3]{HNY13}, see also \cite[\S9, \S11]{KY20}. The sheaf $\cE_{\chi}$ is a rank $n$ local system on $\bA^1$ with wild ramification at $\infty$ and we prove that this sheaf is isomorphic to a classical Airy sheaf. For simplicity of notation, we assume $n$ is even. The case when $n$ is odd only requires replacing $\frac{n}{2}$ with $\frac{n\pm1}{2}$ and changing the sign at a few places.

\subsection{Classical Airy sheaves}
The following result of Katz describes the Frobenius trace of an Airy sheaf.
\begin{prop}\cite[Theorem 17]{Katz87}\label{p:Airy}
Suppose $n\geq 1$ and $n+1$ is prime to $p$. Let $f(x)\in k[x]$ be a polynomial of degree $n+1$ and 
\begin{equation}\label{eq:Airy}
\cF_f:=\mathrm{NFT}_\psi(\cL_{\psi(f)})
\end{equation}
be the associated Airy sheaf. For a finite extension $E/k$ and $t\in E$, the Frobenius trace of $\cF_f$ at $t$ is given by
	\begin{equation}\label{eq:Airy sum}
	\mathrm{Tr}(\cF_f)(t)=-\sum_{x\in E}(\psi\circ\mathrm{Trace}_{E/k})(f(x)+tx).
	\end{equation}
\end{prop}
To identify the eigenvalue $\cE$ with an Airy sheaf we will compute its Frobenius trace by studying the character $\chi : J \to k \xrightarrow{\psi} \bQlt$ explicitly.

\subsection{Explicit extension of the generic functional}
Recall that the character 
\[
\chi:J=S(1)I(1+n/2)\ra S(1)I(1+n/2)/I(2+n) \to k
\] 
is an extension of the linear functional $\phi:I(1+n/2)\ra I(1+n/2)/I(2+n)\ra k$. We study how to write down this extension in terms of matrices.

\subsubsection{}
Let $E_1:=\sum_{i=1}^{n-1}E_{i,i+1}+E_{n,1}$. Thus, $E_r:=E_1^r=\sum_{\Ht(\alpha)=r}E_\alpha+\sum_{\Ht(\alpha)=i-n}E_\alpha$ for $1\leq r\leq n-1$ and $E_1^n=\Id$. Recall that $X_{-1}=E_1^{-1}$, and $\fz=\fz_{\fg}(X_{-1})=\bigoplus_{r\in\bZ/n\bZ}\fz_r$. For $\GL_n$, there is $\fz_r=kE_r$.

Now take a field extension $E/F$ where $F=k(\!(t)\!)$, $E=k(\!(u)\!)$, $t=u^n$. Consider the morphism
\begin{equation}
\sigma:\fg(\!(t)\!)\ra\fg(\!(u)\!), \quad t^rX\mapsto u^{rn}\Ad_{\check{\rho}(u)}X,\ X\in\fg,
\end{equation}

We have 
$$
\sigma(\Lie(I(r)))=\prod_{j\geq r}u^j\fg_j,\quad \forall r\geq 0.
$$ 
Recall that $\fs=\Lie(S)$, $\fs(r)=\Lie(S(r))$. We also have 
$$
\sigma(\fs(r))=\prod_{j\geq r}u^j\fz_j, \quad \forall r\geq 0.
$$

Consider the automorphism $\gamma\in\Aut(\fg(\!(u)\!))$ defined by $\gamma|_{\fg}=\Aut_{\check{\rho}(\zeta_n)}$, $\gamma(u)=\zeta_n^{-1}u$. Thus $\fg(\!(u\!)^\gamma=\fg(\!(u\!)\cap\prod_{r\in\bZ}u^r\fg_r$. 

\begin{lem}\label{l:product}
For any $u^i X\in u^i\fg_i, u^j Y\in u^j\fg_j$ (resp. $u^i X\in u^i\fz_i, u^j Y\in u^j\fz_j$), we have $u^{i+j}XY\in u^{i+j}\fg_{i+j}$ (resp. $u^{i+j}XY\in u^{i+j}\fz_{i+j}$). Here $XY$ is the matrix product in $\fg=\mathrm{Mat}_{n\times n}(k)$.
\end{lem}
\begin{proof}
For $X\in\fg_i, Y\in\fg_j$, this is the same as $\Ad_{\check{\rho}(\zeta_n)}X=\zeta_n^i X$, $\Ad_{\check{\rho}(\zeta_n)}Y=\zeta_n^j Y$. Thus $\Ad_{\check{\rho}(\zeta_n)}XY=\Ad_{\check{\rho}(\zeta_n)}X\Ad_{\check{\rho}(\zeta_n)}Y=\zeta_n^{i+j}XY$, $XY\in\fg_{i+j}$.

For $X\in\fz_i, Y\in\fz_j$, recall that $\fz_r=kE_r=kE_1^r$. Thus $X=c_1E_1^i, Y=c_2E_1^j$, $c_1,c_2\in k$. We get $XY=c_1c_2E_1^{i+j}\in\fz_{i+j}$.
\end{proof}

Denote $Z_r:=\sigma^{-1}(u^r E_r)\in\fs(r))$, $r\geq 1$. Note that 
$$
S(r)=\Id+\prod_{j\geq r}kZ_j.
$$

\begin{lem}\label{l:section}
The composition $H:=\Id+\sum_{r=1}^{1+n}kZ_r\hookrightarrow S(1)\twoheadrightarrow S(1)/S(2+n)$ is an isomorphism.
\end{lem}
\begin{proof}
Surjectivity: give any $s=\Id+X_1+X_2\in S(1)$ where $X_1\in\bigoplus_{1\leq r\leq 1+n}kZ_r$, $X_2\in\prod_{r\geq 2+n}kZ_r$. It suffices to find $s'=\Id+X_3\in S(2+n)$ where $X_3\in\prod_{r\leq 2+n}kZ_r$ such that $(\Id+X_1)(\Id+X_3)=\Id+X_1+X_2$. Applying $\sigma$ to this equality and denote $Y_i:=\sigma(X_i)$, we need to find $Y_3$ such that
$$
(\Id+Y_1)Y_3=Y_2.
$$
Since $Y_1\in\bigoplus_{1\leq r\leq 1+n}kE_r$, we see from Lemma \ref{l:product} that $(\Id+Y_1)^{-1}\in\Id+\prod_{r\geq 1}k Z_r$ and $(\Id+Y_1)^{-1}Y_2\in\Id+\prod_{r\geq 2+n}kE_r$. Thus $Y_3=(\Id+Y_1)^{-1}Y_2$ gives the desired element.

Injectivity: Suppose $X_1, X_2\in\sum_{r=1}^{1+n}kZ_r$, $X_3\in\prod_{r\geq 2+n}kZ_r$ such that $(\Id+X_1)=(\Id+X_2)(\Id+X_3)$. Applying $\sigma$ to the equality and denote $Y_i=\sigma(X_i)$, we get
$$
Y_1=Y_2+(Y_3+Y_2Y_3)
$$ 
where $Y_1,Y_2\in\bigoplus_{r=1}^{1+n}kE_r$, $Y_3+Y_2Y_3\in\prod_{r\geq 2+n}kE_r$. Thus $Y_1=Y_2$, $X_1=X_2$.
\end{proof}

\subsubsection{}
Note that the quotient group $S(1)/S(2+n)$ is unipotent, since it has a filtration $S(1)/S(2+n)\supset S(2)/S(2+n)\supset\cdots\supset S(1+n)/S(2+n)$ such that all the subquotients are $S(r)/S(1+r)\simeq\bGa$. Also, it is commutative. Thus the exponential map induces a group isomorphism
\begin{equation}\label{eq:exp}
\exp: \fs(1)/\fs(2+n)\simeq\sum_{r=1}^{1+n}kZ_r\xrightarrow{\sim}S(1)/S(2+n), \quad X\mapsto \Id+\sum_{r=1}^{1+n}\frac{1}{r!}X^r.
\end{equation}
Its inverse is given by
\begin{equation}\label{eq:log}
\log: \Id+\sum_{r=1}^{1+n}kZ_r\simeq S(1)/S(2+n)\ra \fs(1)/\fs(2+n),\quad \Id+X\mapsto \sum_{r=1}^{1+n}\frac{(-1)^{r+1}}{r}X^r.
\end{equation}

Similarly the exponential map induces a group isomorphism $\fs(1+n/2)/\fs(2+n)\simeq\sum_{r=1+n/2}^{1+n}kZ_r\xrightarrow{\sim}S(1+n/2)/S(2+n)$. The restriction to $\fs(1+n/2)$ of $\phi$ is given by $(\phi\circ\exp)(\sum_{r=1+n/2}^{1+n}y_rZ_r)=(\phi\circ\exp)(y_{1+n}Z_{1+n})=ny_{1+n}$. Thus, its extension $(\chi\circ\exp):\fs(1)/\fs(2+n)\ra k$ is given by
$$
(\chi\circ\exp)(\sum_{r=1}^{1+n}y_rZ_r)=\sum_{r=1}^{n/2}\lambda_r y_r+ny_{1+n}.
$$ 

Assume $\log(\Id+\sum_{r=1}^{1+n}x_rZ_r)=\sum_{r=1}^{1+n}y_rZ_r\in\fs(1)/\fs(2+n)$. From \eqref{eq:log}, we can see
\begin{equation}\label{eq:y_r}
y_r=\sum_{k=1}^r\sum_{i_1+\cdots+i_k=r}c_{i_1,...,i_k}x_{i_1}x_{i_2}\cdots x_{i_k}
\end{equation}
for some fixed coefficients $c_{i_1,...,i_k}$.

From above discussion, we obtain:
\begin{prop}\label{p:mu}
Any extension $\chi$ of $\phi$ from $I(1+n/2)$ to $J=S(1)I(1+n/2)$ is determined by 
\begin{equation}\label{eq:mu}
\chi(\Id+\sum_{r=1}^{1+n}x_rZ_r)=\sum_{r=1}^{n/2}\lambda_r y_r+ny_{1+n}
\end{equation}
where the $\lambda_r$ can be arbitrary constants and the $y_r$'s are polynomials in $x_r$'s as in \eqref{eq:y_r}, independent from any choice.
\end{prop}

\subsection{The relevant part of the Hecke correspondence and the Hecke eigenvalue}
\subsubsection{}
Since $G=\GL_n$, the eigenvalue of the Hecke eigensheaf $\cA_\chi$ is given by the Hecke correspondence associated to the standard representation (equivalently, the first fundamental coweight $\omega_1$). We first have the same diagram as in \cite[\S5.3.1]{KY20}:
\[
\begin{tikzcd}
&  \text{Hecke}_{\omega_1}\arrow{dl}[swap]{\text{pr}_1}  \arrow{dr}{\pr_2}  & \\
\Bun_{J^+}^{0} &       &  \Bun_{J^+}^{1}\times\bA^1.
\end{tikzcd}
\]

Next, let $\star_0\in\Bun_{J^+}^0$ be the trivial bundle, $\star_1\in\Bun_{J^+}^1$ be the bundle corresponding to the minuscule coweight $\diag(1,1,...,1,t^{-1})$ (since $t$ is the coordinate at $\infty$, the degree of the underlying vector bundle is computed by $\mathrm{ord}_{t^{-1}}\circ\det$). We restrict $\pr_2$ to $\star_1\times\bA^1\simeq\bA^1$. The action of $J$ on $\star_0$ induces an embedding of locally closed substack $[L^-G\cap J\backslash J/J^+]$, where $L^-G\cap J=U_{\geq 1+n/2}(k)$ is in the kernel of $\chi$. We restrict $\pr_1$ to this substack. Also, $\chi$ descends to a morphism $[L^-G\cap J\backslash J/J^+]\ra\bGa$. Denote the restriction of $\mathrm{Hecke}_{\omega_1}$ by $\GR_{\omega_1}^*$ and let $p_1:=\chi\circ\pr_1$, $p_2=\pr_2$, we get diagram

\begin{equation}\label{eq:relevant Hecke}
\begin{tikzcd}
& \GR_{\omega_1}^*\arrow{dl}[swap]{p_1}  \arrow{dr}{p_2} & \\
\bGa &  & \bA^1.
\end{tikzcd}
\end{equation}

The Hecke eigenvalue is given by
\begin{equation}\label{eq:cE_mu}
\cE_\chi=p_{2,!}(p_1^*\cL_\psi[n-1]).
\end{equation}

\subsubsection{}
Following the same method as in \cite[\S3]{HNY13} we we obtain the following explicit description of the above diagram.
\begin{prop}\label{p:relevant Hecke}
We have $\GR_{\omega}^*\simeq\bA^{1+n/2}$. For $(m_1,m_2,...,m_{1+n/2})\in\GR_{\omega}^*$, the morphisms $p_1, p_2$ are given by
\begin{align*}
p_1(m_1,...,m_{1+n/2})&=\chi(\Id+\sum_{r=1}^{n/2}m_1^rZ_r)-m_1^{n+1}-\sum_{r=2}^{1+n/2}m_rm_1^{n/2-r+2};\\
p_2(m_1,...,m_{1+n/2})&=-\sum_{r=2}^{1+n/2}m_rm_1^{n/2-r+1}.
\end{align*}
\end{prop}

\subsubsection{}
We make $\chi(\Id+\sum_{r=1}^{n/2}m_1^rZ_r)$ more explicit. First, we compute $\log(\Id+\sum_{r=1}^{n/2}m_1^rZ_r)\in\fs(1)/\fs(2+n)$. Since $(m_1^iZ_i)(m_1^jZ_j)=m_1^{i+j}Z_{i+j}$, we can deduce from $(1-m_1)(1+\sum_{r=1}^{n/2}m_1^r)=1-m_1^{1+n/2}$ that
\begin{align*}
\log(\Id+\sum_{r=1}^{n/2}m_1^rZ_r)&=\log(1-m_1^{1+n/2}Z_{1+n/2})-\log(1-m_1Z_1)\\
&=\sum_{r=1}^{1+n}\frac{m_1^r}{r}Z_r-m_1^{1+n/2}Z_{1+n/2}\in\fs(1)/\fs(2+n).
\end{align*}

Combining the above with \eqref{eq:mu}, we get
\begin{equation}\label{eq:mu of m_1}
\chi(\Id+\sum_{r=1}^{n/2}m_1^rZ_r)=(\chi\circ\exp)\log(\Id+\sum_{r=1}^{n/2}m_1^rZ_r)=\sum_{r=1}^{n/2}\frac{\lambda_r}{r}m_1^r+\frac{n}{1+n}m_1^{1+n}.
\end{equation}

\subsubsection{}
Combining \eqref{eq:cE_mu}, Proposition \ref{p:relevant Hecke} and \eqref{eq:mu of m_1} we can compute the Frobenius trace of the eigenvalue local system $\cE_\chi$.
\begin{prop}\label{p:Frob trace of cE_mu}
The Frobenius trace of $\cE_\chi$ at $a\in k=\bF_q$ is given by
$$
\mathrm{Tr} \cE_\chi(a)=-q^{n/2-1}\sum_{m_1\in k}\psi(f(m_1)+m_1a)
$$
where $f(m_1)=-\frac{1}{n+1}m_1^{n+1}+\sum_{r=1}^{n/2}\frac{\lambda_r}{r}m_1^r\in k[m_1]$ is a degree $n+1$ polynomial.
\end{prop}
\begin{proof}
Recall that $n$ is even. We have
\begin{align*}
\mathrm{Tr}\cE_\chi(a)
=&(-1)^{n-1}\sum_{p_2(m_1,...,m_{n/2+1})=a}\psi(p_1(m_1,...,m_{n/2+1}))\\
=&-\sum_{p_2(m_1,...,m_{n/2+1})=a}\psi(\chi(\Id+\sum_{r=1}^{n/2}m_1^rZ_r)-m_1^{n+1}+m_1p_2(m_1,...,m_{n/2+1}))\\
=&-\sum_{m_1,...,m_{n/2}\in k}\psi(\frac{n}{n+1}m_1^{n+1}+\sum_{r=1}^{n/2}\frac{\lambda_r}{r}m_1^r-m_1^{n+1}+m_1a)\\
=&-q^{n/2-1}\sum_{m_1\in k}\psi(f(m_1)+m_1a).
\end{align*}
For the third equality in the above, we use that $m_{n/2+1}$ is linear with constant coefficient $-1$ in $p_2(m_1,...,m_{n/2+1})$.
\end{proof}

Comparing this to the Frobenius trace of Airy sheaves in Proposition \ref{p:Airy} we obtain the following. 

\begin{cor} Let $f$ be the polynomial from above. Up to Tate twist by $\bQl(\frac{n}{2}-1)$ the eigensheaf $\cE_\chi$ is isomorphic to the Airy sheaf $\cF_{f}$.
\end{cor}

\begin{rem}\label{rem:extension}
From above Proposition \ref{p:mu}, we can see that different extensions $\chi$ correspond to different coefficients $\lambda_r$'s, which correspond to different polynomials $f$ in Proposition \ref{p:Frob trace of cE_mu} that determine the eigenvalue Airy sheaves $\cF_f$. Since the isomorphic classes of Airy sheaves are parametrized by this polynomial, we conclude different extensions $\chi$ of $\phi$ in Airy automorphic datum $(J,\chi)$ have non-isomorphic eigenvalue Airy sheaves. Note that we cannot obtain arbitrary degree $n+1$ polynomial $f$ in Proposition \ref{p:Frob trace of cE_mu}, thus we don't get all the Airy sheaves from Airy automorphic data.
\end{rem}

\appendix
\section{On dominant minuscule coweights}
Recall that a coweight $\mu\in X_*(T)$ is minuscule if it is dominant and satisfies $\theta(\mu)\leq1$. We study some of its properties.

\subsection{Minuscule coweights and parabolic subgroup}\label{a ss:parabolic}
Let $\mu\in M$ be a minuscule coweight. Since $-1\leq\alpha(\mu)\leq 1$ for any $\alpha\in\Phi$ we have the following grading 
$$
\fg=\fg_\mu(-1)\oplus\fg_\mu(0)\oplus\fg_\mu(1)
$$
on $\fg$. More specifically if $\mu \in M$ is a minuscule coweight and if there is a root that does not vanish on $\mu$, then there exists a unique simple root $\alpha_\mu\in\Delta$ such that $\lan \alpha_\mu, \mu \ran=1$ and $\lan\alpha,\mu\ran=0$ for any $\alpha\in\Delta-\{\alpha_\mu \}$. Such simple roots are characterized by the property that they have multiplicity one in the highest root \cite[Chapter VI, Exercise 1.24.]{Bourbaki}. Then $\Phi(\fg_\mu(0))$ are those roots where $\alpha_\mu$ has coefficient $0$, $\Phi(\fg_\mu(1))$ are the roots where $\alpha_\mu$ has coefficient $1$ and $\Phi(\fg_\mu(-1))$ are the roots where $\alpha_\mu$ has coefficient $-1$.

On the other hand, we can associate a parabolic subgroup $P=P_\mu\subset G$ to $\mu$ by defining
$$
P=\left\langle T,U_\alpha \mid \langle \alpha, \mu\rangle \geq0 \right\rangle.
$$
It has Levi decomposition $P=M_P U_P$.  Define $\fm_P=\Lie(M_P)$, $\fu_P=\Lie(U_P)$. We also have the opposite parabolic $P^-$ with unipotent radical $U_P^-$, $\fu_P^-=\Lie(U_P^-)$. By definition we have
$$
\fm_P=\fg_\mu(0),\quad \fu_P=\fg_\mu(1),\quad \fu_P^-=\fg_\mu(-1).
$$
From the above construction we see that the Lie algebras $\fu_P^\pm$ are abelian. 

Consider the Borel subgroup $B_M=B\cap M_P$ of $M_P$ and its (opposite) unipotent radical $U_M=U\cap M_P$ with Lie algebra $\fu_M$ ($U_M^-=U^-\cap M_P$ and $\fu_M^-$ defined similarly). By the above discussion, the simple roots of $M_P$ are $\Delta(M_P)=\Delta-\{\alpha_\mu\}$. Let
\begin{equation}\label{eq:u_mu}
\fu_\mu=\fu_M\oplus\fu_P^-,\quad \fu_\mu^-=\fu_M^-\oplus\fu_P.
\end{equation}

\begin{lem}\label{l:u conju u_mu}
	Let $w_0\in W$ be the longest element of Weyl group of $G$ and $w_{P,0}\in N_{M_P}(T)/T$ the longest element of Weyl group of $M_P$. Denote $w_P:=w_{P,0}w_0$. We have $\Ad_{w_P}\fu=\fu_\mu$, $\Ad_{w_P}\fu^-=\fu_\mu^-$.
\end{lem}
\begin{proof}
We know $\Ad_{w_0}(\fu)=\fu^-=\fu_M^-\oplus\fu_P^-$. Since $w_{P,0}\in M_P$, it normalizes $U_P^{\pm}$ and swaps $U_M^{\pm}$. The lemma follows.
\end{proof}

From above, we see $\fu_\mu$ and $\fu_\mu^-$ are both nilpotent radical of Borel subalgebras of $\fg$. Thus $U_\mu:=\exp(\fu_\mu)$, $U_\mu^-=\exp(\fu_\mu^-)$ are unipotent radicals of Borel subgroups of $G$.

\subsection{Minuscule coweights and cyclic grading}\label{a ss:cyclic grading}
Recall $X_{-1}=N^-+E_\theta\in\fg_{-1}$.
\begin{lem}\label{l:w_P0w_0}\mbox{}
\begin{itemize}
	\item [(1)] $w_{P,0}\theta=\alpha_\mu$, $w_{P,0}\alpha_\mu=\theta$.
	\item [(2)] We can choose a representative of $w_P\in W$ such that $\Ad_{w_P}X_{-1}=X_{-1}$.
	\item [(3)] $w_P\Phi(\fg_r)=\Phi(\fg_r)$, $\forall\ 0\leq r\leq h-1$.
\end{itemize}
\end{lem}
\begin{proof}\ \\
(1): Since $\alpha_\mu$ has multiplicity $1$ in $\theta$, $\theta\in\Phi(\fu_P)$. Since $w_{P,0}$ acts on $U_P$, we have $w_{P,0}\theta\in\Phi(\fu_P)$. Let $\theta=\alpha_\mu+\sum_{\alpha\in\Delta-\{\alpha_\mu\}}m_\alpha \alpha$, $\sum m_\alpha=h-2$. We get
$$
w_{P,0}\theta=(w_{P,0}\alpha_\mu)+\sum_{\alpha\in\Delta-\{\alpha_\mu\}}m_\alpha(w_{P,0}\alpha)>0
$$
where $w_{P,0}\alpha<0$ since $\alpha\in\Phi(\fu_M)$. Thus
$$
1\leq\Ht(w_{P,0}\theta)\leq\Ht(w_{P,0}\alpha_\mu)+\sum (-m_\alpha)=\Ht(w_{P,0}\alpha_\mu)-(h-2)
$$

Thus $\Ht(w_{P,0}\alpha_\mu)=h-1$, $w_{P,0}\alpha_\mu=\theta$. Since $w_{P,0}^2=1$ as it is longest element in Weyl group of $M_P$, we also get $w_{P,0}\theta=\alpha_\mu$. 

(2): We know $w_0\Delta=-\Delta$ and $w_0\theta=-\theta$. Thus
$$
\Ad_{w_0}X_{-1}=\Ad_{w_0}(N^-+E_\theta)=\sum_{\alpha\in\Delta}E_\alpha'+E_{-\theta}'
$$
where $0\neq E_\alpha'\in\fg_\alpha$ and $0\neq E_{-\theta}'\in\fg_{-\theta}$. Similarly, $w_{P,0}(\Delta-\{\alpha_\mu\})=-(\Delta-\{\alpha_\mu\})$. Combining this with part (1), we get
$$
\Ad_{w_P}X_{-1}=\Ad_{w_{P,0}}(\sum_{\alpha\in\Delta-\{\alpha_\mu\}}E_\alpha')+\Ad_{w_{P,0}}E_{\alpha_\mu}'+\Ad_{w_{P,0}}E_{-\theta}'=\sum_{\alpha\in\Delta}E_{-\alpha}'+E_\theta'
$$
which only differs from $X_{-1}$ with nonzero coefficients. We can replace $w_p$ with $tw_p$ for some $t\in T$, such that $E_{-\alpha}'=E_{-\alpha}$, $\forall\ \alpha\in\Delta$. Then $\Ad_{w_P}X_{-1}=N^-+cE_\theta$ for some $c\neq0$. 

Consider the restriction to $\fg_{-1}$ of a degree $h$ invariant polynomial $f\in k[\fg]^G$, which is of the form $f|_{\fg_{-1}}=E_\theta^*\prod_{\alpha\in\Delta}(E_{-\alpha}^*)^{m_\alpha}$, $m_\alpha$ the coefficient of $\alpha$ in $\theta$. Then $1=f(X_{-1})=f(\Ad_{w_P}X_{-1})=c$, $c=1$, $\Ad_{w_P}X_{-1}=X_{-1}$.

(3): Claim: $w_P\alpha\in\Phi(\fg_1)$, $\forall\alpha\in\Delta$. If so, given any $\gamma=\sum_{\alpha\in\Delta}n_\alpha \alpha\in\Phi(\fg_r)$, i.e. $\sum_\alpha n_\alpha\equiv r\mod h$. Thus
$$
\Ht(w_P\gamma)=\sum_\alpha n_\alpha\Ht(w_P\alpha)\equiv\sum_\alpha n_\alpha\equiv r\mod h.
$$

We prove the claim by showing that both $w_{P,0}$ and $w_0$ map $\Phi(\fg_r)$ to $\Phi(\fg_{-r})$. Then, as $w_P=w_{P,0}w_0$, the claim follows. First, since $w_0\Delta=-\Delta$, we get $w_0\Phi(\fg_r)=\Phi(\fg_{-r})$ by the same argument as above. Next, recall $w_{P,0}(\Delta-\{\alpha_\mu\})=-(\Delta-\{\alpha_\mu\})$ and $w_{P,0}\alpha_\mu=\theta\in\Phi(\fg_{-1})$ from part (1). Again by the same argument, we get $w_{P,0}\Phi(\fg_r)=\Phi(\fg_{-r})$, which completes the proof.
\end{proof}

\subsection{Appearance of $\fu_\mu$}\label{a ss:stab=u_mu}
Recall that $U_\mu=\exp(\fu_\mu)$.

\begin{lem}\label{l:stab=u_mu}
	For any minuscule coweight $\mu$, $\Ad_{t^{-\mu}}L^-G\cap I(1)=\exp(\fu_M\oplus\fu_P^-t)=\Ad_{t^{-\mu}}U_\mu$.
\end{lem}
\begin{proof}
The second equality is clear from the fact that $\fu_M\subset\fg_\mu(0)$, $\fu_P^-=\fg_\mu(-1)$. For the first equality, if $\Ad_{t^{-\mu}}U_\alpha(kt^{-m})=U_\alpha(kt^{-m-\alpha(\mu)})\subset I(1)$, $m\geq0$. We have
$$
\begin{cases}
\alpha>0,\ -m-\alpha(\mu)\geq0\Rightarrow m=\alpha(\mu)=0,\ \alpha\in\Phi(\fu_M);\\
\alpha<0,\ -m-\alpha(\mu)\geq 1\Rightarrow m=0, \alpha(\mu)=-1,\ \alpha\in\Phi(\fu_P^-).
\end{cases}
$$

This proves the first equality.
\end{proof}

\subsection{Generic opposite Schubert cells on affine flag varieties} The following should be well-known.  We provide a proof here for the convenience of the reader. Recall that we have decomposition
$$
\mathrm{Fl}_G=[LG/I]=\bigsqcup_{\mu\in X_*(T)}L^-G t^\mu I/I
$$
and that $M = \{ \mu \in X_*(T) \mid 0 \le \langle \alpha, \mu \rangle \le 1\, \forall \alpha \in \Phi^+\}$.
\begin{lem}\label{l:minuscule orbits on Fl_G}
	For any orbit $L^-G t^\mu I/I$ where $\mu\in X_*(T)$ we have $\dim\Stab_{L^-G}(t^\mu)\geq\dim B$. The equality holds if and only if $\mu$ is a dominant minuscule coweight. 
\end{lem}
\begin{proof}
First we have
$$
\Stab_{L^-G}(t^\mu)\simeq\Ad_{t^{-\mu}}L^-G\cap I=T(\Ad_{t^{-\mu}}L^-G\cap I(1))
$$
Thus we obtain
\begin{equation}\label{eq:stab Fl_G}
\begin{split}
\dim\Stab_{L^-G}(t^\mu)&=\dim T+ \#\{\tilde{\alpha}=\alpha+n_\alpha\in\Phi^\aff(G)\mid\tilde{\alpha}(\check{\rho}/h)>0,\ \alpha(\mu)+n_\alpha\leq0 \}\\
&=\dim T+\sum_{\alpha\in\Phi(G)}\#\left\{r\in\bZ\mid\alpha(\mu)\leq r\leq\frac{\Ht(\alpha)-1}{h} \right\}.
\end{split}
\end{equation}
Define $N_\alpha:=\#\{r\in\bZ\mid\alpha(\mu)\leq r\leq\frac{\Ht(\alpha)-1}{h} \}$. For any positive root $\alpha$ we consider the following cases
\[
\begin{cases} \langle \alpha, \mu \rangle=0 \\
\langle \alpha, \mu \rangle\geq1 \\
\langle \alpha, \mu \rangle\leq -1 
\end{cases}
\]
separetely. In the first case $\langle \alpha, \mu \rangle\leq 0\leq\frac{\Ht(\alpha)-1}{h}<1$ and thus $N_\alpha= 1$ and by definition $N_{-\alpha}=0$. In the second case $\langle -\alpha, \mu \rangle \leq -1\leq\frac{-\Ht(\alpha)-1}{h}<0$ and therefore $N_{-\alpha}=\langle \alpha, \mu \rangle\geq 1$ and $N_{\alpha}=0$. Finally in the third case $\langle \alpha, \mu \rangle\leq -1\leq 0\leq\frac{\Ht(\alpha)-1}{h}$ and hence $N_\alpha\geq 2, N_{-\alpha}=0$. Thus for each root $\alpha$ the numbers $N_{\alpha}$ and $N_{-\alpha}$ cannot both be non-zero and summing them up shows that 
\[ \sum N_{\alpha} \ge \Phi^{+} = \dim U.
\]
Moreover if the equality is obtained then the third case above cannot appear, i.e. $\langle \alpha, \mu \rangle\geq 0$ for any positive root $\alpha$ and hence $\mu$ is dominant. From the second case above we have $\langle \alpha, \mu \rangle=1$ for any $\langle \alpha, \mu \rangle\geq 1$. Thus $0\leq\langle \alpha, \mu \rangle\leq 1$ for any positive root $\alpha$.
Conversely when $\mu$ is minuscule, the third case does not occur. The sum over all positive roots $\alpha$ of $N_\alpha$ for $\langle \alpha, \mu \rangle=0$ and $N_{-\alpha}$ for $\langle \alpha, \mu \rangle=1$ is already $\dim U$. We have seen before that for a negative root $\alpha<0$ with $\langle \alpha, \mu \rangle=0$ or positive root $\alpha$ with $\langle \alpha, \mu \rangle=1$ we have $N_\alpha=0$. Thus equality is obtained in this case.
\end{proof}

Since the dimension of stabilizers is a semi-continuous function we obtain the following. 
\begin{cor}\label{c:open orbits on Fl_G}
	$L^-G t^\mu I/I\subset\mathrm{Fl}_G$ is open for any minuscule coweight $\mu$.
\end{cor}

\end{document}